\RequirePackage{fix-cm}
\documentclass[11pt]{article}
\usepackage{graphicx,amsmath,amsfonts,amscd,amssymb,bm,cite,epsfig,epsf,url,color,algorithm,algorithmic,enumerate}
\usepackage{fullpage}
\usepackage[small,bf]{caption}
\usepackage{subfigure}
\usepackage{lscape}
\usepackage[space]{grffile}
\usepackage{epsfig,psfrag,amstext,subfigure,subeqnarray,nccfoots,color,multirow,bm}

\newtheorem{theorem}{Theorem}[section]
\newtheorem{lemma}[theorem]{Lemma}

\newtheorem{corollary}[theorem]{Corollary}
\newtheorem{proposition}[theorem]{Proposition}
\newtheorem{definition}[theorem]{Definition}

\newtheorem{assumption}[theorem]{Assumption}

\usepackage{booktabs}       
\usepackage{pifont}
\newcommand{\cmark}{\ding{52}}%
\newcommand{\xmark}{\ding{55}}%
\usepackage{cite}
\usepackage{hyperref}
\usepackage{graphicx,amsmath,amsfonts,amscd,amssymb,bm,cite,epsfig,epsf,url,color,algorithm,algorithmic,enumerate}
\usepackage{cleveref}

\DeclareMathOperator{\argmin}{\mbox{argmin}}

\def\argmin{\mathop{\rm argmin}}

\numberwithin{equation}{section}

\def \endprf{\hfill {\vrule height6pt width6pt depth0pt}\medskip}
\newenvironment{proof}{\noindent {\bf Proof} }{\endprf\par}

\author{
Jiawei Zhang
\thanks{Shenzhen Research Institute of Big Data, School of Science and Engineering, The Chinese University of Hong Kong, Shenzhen, China.
	Email: 216019001@link.cuhk.edu.cn.}
\and
Peijun Xiao
\thanks{Coordinated Science Laboratory, Department of ISE, University of Illinois at Urbana-Champaign, Urbana, IL. 
	Email: peijunx2@illinois.edu.}
\and
Ruoyu Sun
\thanks{Coordinated Science Laboratory, Department of ISE, University of Illinois at Urbana-Champaign, Urbana, IL. 
	Email: ruoyus@illinois.edu.}
\and
Zhi-Quan Luo
\thanks{Shenzhen Research Institute of Big Data, School of Science and Engineering, The Chinese University of Hong Kong, Shenzhen, China.
	Email: luozq@cuhk.edu.cn.}
}

\title{\bf A Single-Loop Smoothed Gradient Descent-Ascent Algorithm for Nonconvex-Concave Min-Max Problems}

\begin{document}
\maketitle
\begin{abstract}
Nonconvex-concave min-max problem arises in many machine learning applications including minimizing a pointwise maximum of a set of nonconvex functions and robust adversarial training of neural networks. A popular approach to solve this problem is the gradient descent-ascent (GDA) algorithm which unfortunately can exhibit oscillation in case of nonconvexity. In this paper, we introduce  a ``smoothing" scheme which can be combined with GDA to stabilize the oscillation and ensure convergence to a stationary solution. We prove that the stabilized GDA algorithm can achieve an $O(1/\epsilon^2)$ iteration complexity for minimizing the pointwise maximum of a finite collection of nonconvex functions.
Moreover, the smoothed GDA algorithm achieves an $O(1/\epsilon^4)$ iteration complexity for general nonconvex-concave problems. Extensions of this stabilized GDA algorithm to multi-block cases are presented. To the best of our knowledge, this is the first algorithm to achieve $O(1/\epsilon^2)$ for a class of nonconvex-concave problem. We illustrate the practical efficiency of the stabilized GDA algorithm on robust training.

\end{abstract}

\section{Introduction}
Min-max problems have drawn considerable interest from the machine learning and other engineering communities. They appear in applications such as adversarial learning \cite{goodfellow2014generative, arjovsky2017wasserstein, madry2017towards},
robust optimization \cite{ben2009robust, delage2010distributionally, namkoong2016stochastic, namkoong2017variance}, empirical risk minimization \cite{zhang2017stochastic, tan2018stochastic}, and reinforcement learning \cite{du2017stochastic, dai2017sbeed}.
Concretely speaking, a min-max problem is in the form:
\begin{equation}\label{minimax1}
\min_{x\in X} \max_{y\in Y}f(x, y),
\end{equation}
where $X\subseteq \mathbb{R}^n$ and $Y\in \mathbb{R}^m$ are convex and closed sets and $f$ is a smooth function.
In the literature,  the convex-concave min-max problem, where $f$ is convex in $x$ and concave in $y$, is well-studied \cite{nemirovski2004prox, nesterov2007dual, monteiro2010complexity, palaniappan2016stochastic, gidel2016frank, mertikopoulos2018optimistic, hamedani2018iteration, mokhtari2019unified}.
However, many practical applications involve nonconvexity, and this is the focus of the current paper.
Unlike the convex-concave setting where we can compute the global stationary solution efficiently, to obtain a global optimal solution for the setting where $f$ is nonconvex with respect to $x$ is difficult.

In this paper, we consider the nonconvex-concave min-max problem \eqref{minimax1} where $f$ is nonconvex in $x$ but concave of $y$, as well as a special case in the following form:
\begin{equation}\label{minmax2}
\min_x\max_{y\in Y}F(x)^Ty,
\end{equation}
where $Y=\{(y_1, \cdots, y_m)^T\mid \sum_{i=1}^my_i=1, y_i\ge 0\}$ is a probability simplex and 
$$F(x)=(f_1(x), f_2(x), \cdots, f_m(x))^T$$ is a smooth map from $\mathbb{R}^n$ to $\mathbb{R}^m$.
Note that \eqref{minmax2} is equivalent to the problem of minimizing the point-wise maximum of a finite collection of functions:
\begin{equation}\label{minmax discrete set}
\min_x\max_{1\le i\le m}f_i(x).
\end{equation}
If $f_i(x)=g(x, \xi_i)$ is a loss function or a negative utility function at a data point $\xi_i$, then problem \eqref{minmax discrete set} is to find the best parameter of the worst data points.
This formulation is frequently used in machine learning and other fields. For example, adversarial training \cite{nouiehed2019solving, madry2017towards}, fairness training \cite{nouiehed2019solving} and distribution-agnostic meta-learning  \cite{collins2020distribution} can be formulated as \eqref{minmax discrete set}. We will discuss the formulations for these applications in details in Section \ref{sec: application}.

\begin{table*}[t]
\begin{center}
\begin{small}
\begin{sc}
\begin{tabular}{cccccc}
\toprule
Algorithm & Complexity & Simplicity & Multi-block \\
\midrule
\cite{ostrovskii2020efficient} & $\mathcal{O}(1/\epsilon^{2.5})$ & Triple-loop &  \xmark\\
\cite{lin2020near} & $\mathcal{O}(1/\epsilon^{2.5})$ & Triple-loop &  \xmark\\
\cite{nouiehed2019solving} & $\mathcal{O}(1/\epsilon^{3.5})$ & Double-loop &  \xmark \\
\cite{lu2019hybrid} & $\mathcal{O}(1/\epsilon^4)$ & Single-loop &  \cmark \\
This paper & $\mathcal{O}(1/\epsilon^2)$ & Single-loop &  \cmark\\
\bottomrule
\end{tabular}
\end{sc}
\end{small}
\end{center}
\caption{Comparison of the algorithm in this paper with other works in solving problem (\ref{minmax2}). Our algorithm has a better convergence rate among others, and the single-loop and multi-block design make the algorithm suitable for solving large-scale problems efficiently.}
\label{table: comparison}
\vspace{-5mm}
\end{table*}

Recently, various algorithms have been proposed for nonconvex-concave min-max problems
\cite{thekumparampil2019efficient, jin2019minmax, rafique2018non, nouiehed2019solving, lin2019gradient, lu2019hybrid,ostrovskii2020efficient, lin2020near}.
These algorithms can be classified into three types
 based on the structure: single-loop, double-loop and triple loop.
Here a single-loop algorithm is an iterative algorithm where each iteration step has a closed form update, while a double-loop algorithm uses an iterative algorithm to approximately solve the sub-problem at each iteration. A triple-loop algorithm uses a double-loop algorithm to approximately solve a sub-problem at every iteration.
To find an $\epsilon$-stationary solution, double-loop and tripe-loop algorithms have two main drawbacks.
First, these existing multi-loop algorithms require at least $\mathcal{O}(1/\epsilon^2)$ outer iterations, while the iteration numbers of the other inner loop(s) also depend on $\epsilon$. Thus, the iteration complexity of the existing multi-loop algorithms is more than $\mathcal{O}(1/\epsilon^2)$ for \eqref{minmax2}. Among all the existing algorithms, the best known iteration complexity is $\mathcal{O}(1/\epsilon^{2.5})$ from two triple-loop algorithms \cite{ostrovskii2020efficient, lin2020near}.
Since the best-known lower bound for solving \eqref{minmax2} using first-order algorithms is $\mathcal{O}(1/\epsilon^2)$, so there is a gap between the existing upper bounds and the lower bound.
Another drawback of multi-loop algorithms is
their difficulty in solving problems with multi-block structure,
since the acceleration steps used in their inner loops cannot be easily extended to multi-block cases,
and a standard double-loop algorithm without acceleration can be very slow.
This is unfortunate because the min-max problems with block structure is important for distributed training \cite{lu2019hybrid} in machine learning and signal processing. 

Due to the aforementioned two drawbacks of double-loop and triple-loops algorithms, we focus in this paper on single-loop algorithms
 in hope to achieve the optimal iteration complexity $\mathcal{O}(1/\epsilon^2)$ for the nonconvex-concave problem  \eqref{minmax2}.
Notice that the nonconvex-concave applications in the aforementioned studies  \cite{thekumparampil2019efficient, jin2019minmax, rafique2018non, nouiehed2019solving, lin2019gradient, lu2019hybrid,ostrovskii2020efficient, lin2020near}
can all be formulated as \eqref{minmax2},
although the iteration complexity results derived in these papers are only for general nonconvex-concave problems. In other words,
the structure of \eqref{minmax2} is not used in the theoretical analysis.
One natural question to ask is: \textbf{can we design a single loop algorithm with an iteration complexity lower than $\mathcal{O}(1/\epsilon^{2.5})$ for the min-max problem \eqref{minmax2}?} 

\textbf{Existing Single-loop algorithms.} A simple single-loop algorithm is the so-called Gradient Descent Ascent (GDA)  which alternatively performs gradient descent to the minimization problem and gradient ascent to the maximization problem.
GDA can generate an $\epsilon$-stationary solution for a nonconvex-strongly-concave problem with iteration complexity $\mathcal{O}(1/\epsilon^2)$ \cite{lin2019gradient}.
However, GDA will oscillate with constant stepsizes around the solution if the maximization problem is not strongly concave \cite{mokhtari2019unified}. So the stepsize should be proportional to $\epsilon$ if we want an $\epsilon$-solution.
These limitations slow down GDA which has an $\mathcal{O}(1/\epsilon^5)$ iteration complexity for nonconvex-concave problems.
Another single-loop algorithm \cite{lu2019hybrid} requires diminishing step-sizes to guarantee convergence and its complexity is $\mathcal{O}(1/\epsilon^4)$. 
\cite{xu2020unified} also proposes a single-loop algorithm for min-max problems by performing GDA to a regularized version of the original min-max problem and the regularization term is diminishing. 
The iteration complexity bounds given in the references \cite{lu2019hybrid, lin2019gradient, xu2020unified} are worse than the ones from multi-loop algorithms using acceleration in the subproblems.

%

In this paper, we propose a single-loop ``smoothed gradient descent-ascent'' algorithm with optimal iteration complexity for the nonconvex-concave problem \eqref{minmax2}. 
Inspired by \cite{zhang2020proximal, zhang2022global}, 
to fix the oscillation issue of GDA discussed above, we introduce an exponentially weighted sequence $z^t$ of the primal iteration sequence $x^t$ and include a quadratic proximal term centered at $z^t$ to objective function. Then we perform a GDA step to the proximal function instead of the original objective. With this smoothing technique, an $\mathcal{O}(1/\epsilon^2)$ iteration complexity can be achieved for problem \eqref{minmax2} under mild assumptions. Our contributions are three fold.
\begin{itemize}
    \item \textbf{Optimal order in convergence rate.} We propose a single-loop algorithm \textbf{Smoothed-GDA} for nonconvex-concave problems which finds an $\epsilon$-stationary solution within $\mathcal{O}(1/\epsilon^2)$ iterations for
    problem \eqref{minmax2} under mild assumptions.
     \item \textbf{General convergence results.} The {\bf Smoothed-GDA} algorithm can also be applied to solve general nonconvex-concave problems with an $\mathcal{O}(1/\epsilon^4)$ iteration complexity. This complexity is the same as in \cite{lu2019hybrid}. However, the current algorithm does not require the compactness of the domain $X$, which significantly extends the applicability of the algorithm.
    \item \textbf{Multi-block settings.} We extend the {\bf Smoothed-GDA} algorithm to the multi-block setting and give the same convergence guarantee as the one-block case.
\end{itemize}

The paper is organized as follows. In Section \ref{sec: application}, we describe some applications of nonconvex-concave problem \eqref{minmax2} or \eqref{minmax discrete set}. The details of the {\bf Smoothed-GDA} algorithm as well as the main theoretical results are given in Section \ref{sec: alg}. The proof sketch is given in Section \ref{sec: proof sketch}. The proofs and the details of the numerical experiments are in the appendix.

\section{Representative Applications}\label{sec: application}
We give three application examples which are in the min-max form \eqref{minmax2}.

\vspace{0.8mm}
\textbf{1. Robust learning from multiple distributions.} Suppose the data set is from $n$ distributions: $D_1, \cdots, D_n$. Each $D_i$ is a different perturbed version of the underlying true distribution $D_0$. Robust training is formulated as minimizing the maximum of expected loss over the $n$ distributions as
\begin{align}
   \min_{x \in X} \max_{i} \mathbb{E}_{a\sim D_{i}}[F(x ; a)]=\min _{x \in X} \max _{y \in Y} \sum_{i=1}^{m} y_{i} f_{i}(x),
\end{align}
where $Y$ is a probability simplex, $F(x; a)$ represents the loss with model parameter $x$ on a data sample $a$. Notice that $f_i(x) = \mathbb{E}_{a\sim D_{i}}[F(x ; a)]$ is the expected loss under distribution $D_i$. In adversarial learning \cite{madry2017towards, kurakin2016adversarial, goodfellow2014explaining}, $D_i$ corresponds to the distribution that is used to generate adversarial examples. In Section \ref{sec: experiment}, we will provide a detailed formulation of adversarial learning on the data set MNIST and apply the Smoothed GDA algorithm to this application.

\vspace{0.8mm}
\textbf{2. Fair models}. In machine learning, it is common that the models may be unfair,
i.e. the models might discriminate against individuals based on their membership in some group \cite{hardt2016equality, dwork2012fairness}.
For example, an algorithm for predicting a person's salary might use that person's protected attributes, such as gender, race, and color. Another example is training a logistic regression model for classification which can be biased against certain categories. To promote fairness, \cite{mohri2019agnostic} proposes a framework to minimize the maximum loss incurred by the different categories:
\begin{align}
    \min_{x \in X} \max_{i} f_i(x),
\end{align}
where $x$ represents the model parameters and $f_i$ is the corresponding loss for category $i$.

\vspace{0.8mm}
\textbf{3. Distribution-agnostic meta-learning.} Meta-learning is a field about learning to learn, i.e. to learn the optimal model properties so that the model performance can be improved. One popular choice of meta-learning problem is called gradient-based Model-Agnostic Meta-Learning (MAML) \cite{finn2017model}. The goal of MAML is to learn a good global initialization such that for any new tasks, the model still performs well after one gradient update from the initialization.

One limitation of MAML is that it implicitly assumes the tasks come from a particular distribution, and optimizes the expected or sample average loss over tasks drawn from this distribution. This limitation might lead to arbitrarily bad worst-case performance and unfairness. To mitigate these difficulties, \cite{collins2020distribution} proposed a distribution-agnostic formulation of MAML:
\begin{align}\label{eqn: dis MAML}
    \min_{x\in X} \max_{i} f_i(x - \alpha \nabla f_i(x)).
\end{align}
Here, $f_i$ is the loss function associated with the $i$-th task, $x$ is the parameter taken from the feasible set $X$, and $\alpha$ is the stepsize used in the MAML for the gradient update. Notice that each $f_i$ is still a function over $x$, even though we take one gradient step before evaluating the function. This formulation \eqref{eqn: dis MAML} finds the initial point that minimizes the objective function after one step of gradient over all possible loss functions. It is shown that solving the distribution-agnostic meta-learning problem improves the worst-case performance over that of the original MAML \cite{collins2020distribution} across the tasks.

\section{Smoothed GDA Algorithm and Its Convergence}\label{sec: alg}




Before we introduce the Smoothed-GDA algorithm, we first define the  stationary solution and the $\epsilon$-stationary solution of problem \eqref{minimax1}.
\begin{definition}\label{def: eps stationary solution}
Let $\mathbf{1}_X(x), \mathbf{1}_Y(y)$ be the indicator functions of the sets $X$ and $Y$ respectively. A pair $(x, y)$ is an $\epsilon$-solution set of problem \eqref{minimax1} if there exists a pair $(u, v)$ such that
\begin{align}
u\in \nabla_xf(x, y)+\partial {\mathbf{1}_X(x)},  \quad
v\in-\nabla_yf(x, y)+\partial{\mathbf{1}_Y(y)}, \quad \textrm{and} \quad
\|u\|, \|v\|\le \epsilon,
\end{align}
where $\partial g(\cdot)$ denotes the sub-gradient of a function $g$. A pair $(x, y)$ is a 
game stationary point
if $u = 0, v = 0$.
\end{definition}

\begin{definition}
The projection of a point $y$ onto a set $X$ is defined as $P_X(y) = \argmin_{x \in X} \frac{1}{2} \|x - y\|^2$.
\end{definition}

\subsection{Smoothed Gradient Descent-Ascent (Smoothed-GDA)}
 A simple algorithm for solving min-max problems is the Gradient Descent Ascent (GDA) algorithm (Algorithm \ref{Alg:GDA}),
 which performs a gradient
 descent to the $\min$ problem and a gradient ascent to the $\max$ problem alternatively.
 It is well-known that with constant step size, GDA can oscillate between iterates and fail to converge even for a simple bilinear min-max problem: $\min_{x\in \mathbb{R}^n}\max_{y\in \mathbb{R}^n}x^Ty.$

To fix the oscillation issue, we introduce a ``smoothing''  technique to the primal updates.
Note that smoothing is a common technique in traditional optimization such as Moreau-Yosida smoothing \cite{parikh2014proximal} and Nesterov's smoothing \cite{nesterov2005smooth}.
More concretely, we introduce an auxiliary sequence $\{z^t\}$ and  define a function $K(x, z; y)$ as
\begin{equation}
K(x, z;y)=f(x, y)+\frac{p}{2}\|x-z\|^2,
\end{equation}
where $p>0$ is a constant, and we perform gradient descent and gradient ascent alternatively on this function instead of the original function $f(x, y)$. After performing one-step of GDA to the function $K(x^t, z^t; y^t)$, $z^t$ is updated by an averaging step.
The ``Smoothed GDA'' algorithm is formally presented in Algorithm \ref{Alg2}. Note that our algorithm is different from the one in \cite{xu2020unified}, as \cite{xu2020unified} uses an regularization term $\alpha_t(\|x\|^2 - \|y\|^2)$ and requires this term to diminishing.

\begin{minipage}{0.46\textwidth}
\begin{algorithm}[H]
    \caption{GDA}
    \label{Alg:GDA}
\begin{algorithmic}[1]
\STATE Initialize $x^0, y^0$;
\STATE Choose $c, \alpha > 0$;
\FOR{$t=0,1,2,\ldots,$}
\STATE $x^{t+1}=P_X(x^t-c\nabla_xf(x^t, y^t))$;
\STATE $y^{t+1}=P_Y(y^t+\alpha\nabla_yf(x^{t+1}, y^t))$;
\ENDFOR
\end{algorithmic}
\end{algorithm}
\end{minipage}
\hfill
\begin{minipage}{0.46\textwidth}
\begin{algorithm}[H]
    \caption{Smoothed-GDA}
    \label{Alg2}
\begin{algorithmic}[1]
\STATE Initialize $x^0, z^0, y^0$ and $0< \beta \le 1$.
\FOR{$t=0,1,2,\ldots,$}
\STATE $x^{t+1}=P_X(x^t-c\nabla_xK(x^t, z^t; y^t))$;
\STATE $y^{t+1}=P_Y(y^t+\alpha\nabla_yK(x^{t+1}, z^t;y^t))$;
\STATE $z^{t+1}=z^t+\beta(x^{t+1}-z^t)$,
\ENDFOR
\end{algorithmic}
\end{algorithm}
\end{minipage}

\vspace{3mm}

Notice that when $\beta=1$, Smoothed-GDA is just the standard GDA. Furthermore, if the variable $x$ has a block structure, i.e., $x$ can be decomposed into $N$ blocks as
$$x=(x_1^T, \cdots, x_N^T)^T,$$
then Algorithm \ref{Alg2} can be extended to a multi-block version which we call the Smoothed Block Gradient Descent Ascent (Smoothed-BGDA) Algorithm (see Algorithm  \ref{Alg3}).
In the multi-block version, we update the primal variable blocks alternatingly
and use the same strategy to update the dual variable and the auxiliary variable as in the single-block version.
\begin{algorithm}[ht]
    \caption{Smoothed Block Gradient Descent Ascent (Smoothed-BGDA)}
    \label{Alg3}
\begin{algorithmic}[1]
\STATE Initialize $x^0, z^0, y^0$; 
\FOR{$t=0,1,2,\ldots,$}
\FOR{$i=1,2,\ldots,N$}
\STATE $x^{t+1}_i=P_X(x^t_i-c\nabla_{x_i}K(x^{t+1}_1, x_2^{t+1}, \cdots, x_{i-1}^{t+1}, x_i^t, \cdots, x_N^t, z^t; y^t))$;
\ENDFOR
\STATE $y^{t+1}=P_Y(y^t+\alpha\nabla_yK(x^{t+1}, z^t;y^t))$;
\STATE $z^{t+1}=z^t+\beta(x^{t+1}-z^t)$, where $0< \beta \le 1$;
\ENDFOR
\end{algorithmic}
\end{algorithm}

\subsection{Iteration Complexity for Nonconvex-Concave Problems}
In this subsection, we present the iteration complexities of Algorithm \ref{Alg2} and Algorithm \ref{Alg3} for general
nonconvex-concave problems \eqref{minimax1}. We first state some basic assumptions.
\begin{assumption}\label{basic-ass}
We assume the following.
\begin{enumerate}
\item $f(x, y)$ is smooth and the gradients $\nabla_xf(x, y), \nabla_yf(x, y)$ are $L$-Lipschitz continuous.
\item $Y$ is a closed, convex and compact set of $\mathbb{R}^m$. $X$ is a closed and convex set.
\item The function $\psi(x)=\max_{y\in Y}f(x, y)$ is bounded from below by some finite constant $\underline{f}>-\infty$.
\end{enumerate}
\end{assumption}
\begin{theorem}\label{general}
Consider solving problem \eqref{minimax1} by Algorithm \ref{Alg2} (or Algorithm \ref{Alg3}).
Suppose Assumption~\ref{basic-ass} holds, and we choose the algorithm parameters to satisfy $p>3L,\ c<1/(p+L)$ and
\begin{align}\label{eqn: parameters for Alg 2 general case}
     \alpha<\min\left\{\frac{1}{11L}, \frac{c^2(p-L)^2}{4L(1+c(p-L))^2}\right\}, \beta\le\min\left\{\frac{1}{36}, \frac{(p-L)^2}{384p(p+L)^2}\right\}.
\end{align}
Then, the following holds:
\begin{itemize}
\item
(One-block case) For any integer $T>0$,
if we
further let $\beta<1/\sqrt{T}$, then there exists a $t\in\{1, 2,\cdots, T\}$ such that $(x^{t+1}, y^{t+1})$ is a $\mathcal{O}(T^{-1/4})$-stationary solution. This means we can obtain an $\epsilon$-stationary solution within $\mathcal{O}(\epsilon^{-4})$ iterations.
\item
(Multi-block case) If we replace the condition of $\alpha$ in Algorithm \ref{Alg3} by
\begin{equation}\label{condition}
  \alpha\le\min\left\{\frac{1}{11L}, \frac{c^2(p-L)^2}{4L(1+c(p+L)N^{3/2}+c(p-L))^2}\right\}
\end{equation}
and further require $\beta\le \epsilon^2$,
then we can obtain an $\epsilon$-stationary solution within $\mathcal{O}(\epsilon^{-4})$ iterations of Algorithm \ref{Alg3}.
\end{itemize}
\end{theorem}
%
\noindent{\bf Remark.} The reference \cite{lu2019hybrid} derived the same iteration complexity of $\mathcal{O}(\epsilon^{-4})$ under the additional compactness assumption on $X$. This assumption may not be satisfied for some applications where $X$ can the entire space.

\subsection{The Iteration Complexity of Achieving $\epsilon$-Optimization-Stationary Point}
In this subsection, we consider the optimization stationary point defined in \cite{lin2019gradient}.
\begin{definition}\label{def: optimization stationary point}
A pair $(x, y)$ is a $\epsilon$-optimization-stationary point ($\epsilon > 0$) if $$\|z - x^*(z)\| \leq \epsilon$$ where 
$$x^*(z) = \arg\min_{x\in X}\max_{y\in Y}K(x, z; y).$$
\end{definition}

\begin{theorem}\label{general-for-optimization-stationary-point}
Consider solving problem \eqref{minimax1} by Algorithm \ref{Alg2} (or Algorithm \ref{Alg3}).
Suppose Assumption~\ref{basic-ass} holds, and we choose the algorithm parameters as the same as Theorem \ref{general}.
Then, the complexity to achieve $\epsilon$-stationary point is also  $\mathcal{O}(\epsilon^{-4})$.
\end{theorem}

\noindent{\bf Remark.} To the best of our knowledge, Algorithm \ref{Alg2} (or Algorithm \ref{Alg3}) is the only single-loop algorithm that could achieve this $\mathcal{O}(\epsilon^{-4})$ complexity (Theorem \ref{general-for-optimization-stationary-point}). The proof is given in Appendix \ref{Appendix: new theorem}.  

\subsection{Convergence Results for Minimizing the Point-Wise Maximum of Finite Functions}
Now we state the improved iteration complexity results for the special min-max problem \eqref{minmax2}. We claim that our algorithms (Algorithm \ref{Alg2} and Algorithm \ref{Alg3}) can achieve the optimal order of iteration complexity of $\mathcal{O}(\epsilon^{-2})$ in this case.

For any stationary solution of \eqref{minmax2} denoted as $(x^*, y^*)$, the following KKT conditions hold:
\begin{eqnarray}\label{KKTfororiginal}
&&\nabla F(x^*)y^*=0,\label{stationary}\\
&&\sum_{i=1}^my_i^*=1,\\
&&y_i^*\ge 0,\forall i\in [m]\\
&&\mu-\nu_i=f_i(x^*), \forall i\in[m],\\
&&\nu_i\ge 0, \nu_iy_i^*=0, \forall i\in[m],
\end{eqnarray}
where $\nabla F(x) $ denotes the Jacobian matrix of $F$ at $x$, while $\mu$, $\nu$ are the multipliers for the equality constraint $\sum_{i=1}^my_i=1$ and the inequality constraint $y_i\ge 0$ respectively.


At any stationary solution $(x^*, y^*)$, only the functions $f_i(x^*)$ for any index $i$ with $y^*_i>0$ contribute to the objective function $\sum_{i=1}^Ny_i^*f_i(x^*)$ and they correspond to the worst cases in the robust learning task. In other words, any function $f_i(\cdot)$ with $y^*_i>0$ at $(x^*, y^*)$ contains important information of the solution.
We denote a set $\mathcal{I}_+(y^*)$ to represent the set of indices for which $y_i^*>0$. We will make a mild assumption on this set.
\begin{assumption}\label{assumption for special}
For  any $(x^*, y^*)$ satisfying  \eqref{KKTfororiginal}, we have $\nu_i>0, \forall i\notin\mathcal{I}_+(y^*)$.
\end{assumption}
\textbf{Remark.}
The assumption is called ``strict complementarity'', a common assumption in the field of variation inequality \cite{harker1990finite, facchinei2007finite} which is closely related to the study of min-max problems. This assumption is used in many other optimization papers \cite{ forsgren2002interior, carbonetto2009interior, liang2014local, namkoong2016stochastic, lu2019snap}.
Strict complementarity is generically true (i.e. holds with probability 1) if there is a linear term in the objective function and the data is from a continuous distribution (similar to \cite{zhang2020proximal, lu2019snap}). 
Moreover, we will show that we can prove Theorem \ref{discrete} using a weaker regularity assumption rather than the strict complementarity assumption:
\begin{assumption}\label{regularity in main text}
For any $(x^*, y^*)\in W^*$, the matrix $M(x^*)$ is of full column rank, where

$$M(x^*)=\begin{Bmatrix}
J_{\mathcal{T}(x^*)}&\mathbf{1}
\end{Bmatrix}.$$
\end{assumption}
We say that Assumption \ref{regularity in main text} is weaker since the strict complementarity assumption (Assumption \ref{assumption for special}) can imply Assumption \ref{regularity in main text} according to Lemma \ref{uniqueness} in the appendix.
In the appendix, we will see that Assumption \ref{regularity in main text} holds with probability $1$ for a robust regression problem with a square loss (see Proposition \ref{generic}).

We also make the following common ``bounded level set'' assumption.
\begin{assumption}\label{boundedness}
  The set $\{x\mid \psi(x)\le R\}$ is bounded for any $R>0$. Here $\psi(x)=\max_{y\in Y}f(x, y)$.
\end{assumption}

\noindent{\bf Remark.} This bounded-level-set assumption is to ensure the iterates would stay bounded. Actually, assuming the iterates $\{z^t\}$ are bounded will be enough for our proof.
The bounded level set assumption, a.k.a. coerciveness assumption, is widely used in many papers  \cite{zeng2018nonconvex, cannelli2019asynchronous, wang2019global}.
Bounded-iterates-assumption itself is also common in optimization
\cite{xu2015block, defossez2020convergence, carbonetto2009interior}.
In practice, people usually add a regularizer to the objective function to make the level set and the iterates bounded (see \cite{xie2020maximum} for a neural network example).

\begin{theorem}\label{discrete}
Consider solving problem \ref{minmax2} by Algorithm \ref{Alg2} or Algorithm \ref{Alg3}. Suppose that Assumption \ref{basic-ass}, \ref{assumption for special} holds and either Assumption \ref{boundedness} holds or assume $\{z^t
\}$ is bounded. Then there exist constants $\beta'$ and $\beta''$ (independent of $\epsilon$ and $T$) such that the following holds:
\begin{enumerate}
    \item (One-block case) If we choose the parameters in Algorithm \ref{Alg2} as in \eqref{eqn: parameters for Alg 2 general case} and we further let
$\beta<\beta'$ , then
        \begin{enumerate}
        \item Every limit point of $(x^t, y^t)$ is a solution of \eqref{minmax2}.
        \item The iteration complexity of Algorithm   \ref{Alg2} to obtain an $\epsilon$-stationary solution is $\mathcal{O}(1/\epsilon^2)$.
        \end{enumerate}
    \item (Multi-block case) Consider using Algorithm \ref{Alg3} to solve Problem \ref{minmax2}. If we replace the condition for $\alpha$ in \eqref{eqn: parameters for Alg 2 general case} by \eqref{condition}
and require $\beta$ satisfying  $\beta<\epsilon^2$ and $\beta<\beta''$,
    then we have the same results as in the one-block case.
\end{enumerate}
\end{theorem}

\section{Proof Sketch}\label{sec: proof sketch}
In this section, we give a proof sketch of the main theorem on the one-block cases; the proof details will be given in the appendix.

\subsection{The Potential Function and  Basic Estimates}
To analyze the convergence of the algorithms, we construct a potential function and study its behavior along the iterations.
We first give the intuition why our algorithm works.
We define the dual function $d(\cdot)$ and the proximal function $P(\cdot)$ as
\begin{eqnarray*}
d(y, z)= \min_{x\in X}K(x, z; y),\quad
P(z)= \min_{x\in X}\{\max_{y\in Y}K(x, z; y)\}.
\end{eqnarray*}
We also let
\begin{eqnarray*}
x(y, z)&=&\arg\min_{x\in X}K(x, z; y),\\
x^*(z)&=&\arg\min_{x\in X}\max_{y\in Y}K(x, z; y),\\
y_+^t(z^t)&=&P_Y(y^t+\alpha\nabla_yK(x(y^t, z^t), z^t; y^t)).
\end{eqnarray*}
Notice that by Danskin's Theorem, we have $\nabla_yd(y, z)=\nabla_yK(x(y, z), z;y)$ and $\nabla_zP(z)=p(z-x^*(z))$.
Recall in Algorithm \ref{Alg2}, the update for $x^t, y^t$ and $z^t$ can be respectively viewed as a primal descent for the function $K(x^t, z^t; y^t)$, approximating dual ascent to the dual function  $d(y^t, z^t)$ and approximating proximal descent to the proximal function $P(z^t)$.
We define a potential function as follows:
\begin{equation}\label{potential}
\phi^t=\phi(x^t, y^t, z^t)=K(x^t, z^t; y^t)-2d(y^t, z^t)+2P(z^t),
\end{equation}
which is a linear combination of the primal function $K(\cdot)$, the dual function $d(\cdot)$ and the proximal function $P(\cdot)$.
We hope the potential function decreases after each iteration and is bounded from below. In fact, it is easy to prove that $\phi^t\ge \underline{f}$ for any $t$ (see appendix), but it is harder to prove the decrease of $\phi^t$. Since the ascent for dual and the descent for proximal is approximate, an error term occurs when estimating the decrease of the potential function. Hence, certain error bounds are needed. 

Using some primal error bounds, we have the following basic descent estimate.

\begin{proposition}\label{basic-estimate}
Suppose  the parameters of Algorithm  \ref{Alg2} satisfy \eqref{eqn: parameters for Alg 2 general case}, then
\begin{eqnarray}\label{basic estimate}
\phi^t-\phi^{t+1}
&\ge&\frac{1}{8c}\|x^t-x^{t+1}\|^2 +\frac{1}{8\alpha}\|y^t-y^t_+(z^t)\|^2+\frac{p\beta}{8}\|z^t-x^{t+1}\|^2\\
\label{eqn: negative term}&&-24p\beta\|x^*(z^{t})-x(y^t_+(z^t), z^{t})\|^2.
\end{eqnarray}
\end{proposition}
We would like the potential function $\phi^t$ to decrease sufficiently after each iteration. Concretely speaking, we want to eliminate the negative term (\ref{eqn: negative term}) and show that the following ``sufficient-decrease'' holds for each iteration $t$:
\begin{equation}\label{suff-decrease}
\phi^t-\phi^{t+1}\ge \frac{1}{16c}\|x^t-x^{t+1}\|^2+\frac{1}{16\alpha}\|y^t-y^t_+(z^t)\|^2+\frac{p\beta}{16}\|z^t-x^{t
+1}\|^2.
\end{equation}
It is not hard to prove that if \eqref{suff-decrease} holds for $t\in\{0, 1, \cdots, T-1\}$, then there exists a $t\in\{1, 2, \cdots, T\}$ such that $(x^t, y^t)$ is a $C/\sqrt{T\beta}$-solution for some constant $C>0$.
Moreover, if \eqref{suff-decrease} holds for any $t$, then the iteration complexity is $\mathcal{O}(1/\epsilon^2)$ and we can also prove that every limit point of the iterates is a min-max solution.
Therefore by the above analysis, the most important thing is to bound the term $\|x^*(z^t)-x(y^t_+(z^t), z^t)\|^2$, which is related to the so-call ``dual error bound''.

If $\|y^t-y^t_+(z^t)\|=0$, then $y^t_+(z^t)$ is the maximizer of $d(y, z^t)$ over $y$, and thus $x^*(z^t)$ is the same as $x(y^t_+(z^t), z^{t})$. A natural question is whether we can use the term $\|y^t-y^t_+(z^t)\|$ to bound
$\|x^*(z^{t})-x(y^t_+(z^t), z^{t})\|^2$? The answer is yes, and
we have the following ``dual error bound''.

\begin{lemma}\label{dual error bound}
If Assumptions \ref{basic-ass} and \ref{assumption for special} hold for \eqref{minmax2} and there is an $R>0$ with $\|z^t\|\le R$, then there exists $\delta>0$ such that if
$$\max\{\|x^t-x^{t+1}\|, \|y^t-y^t_+(z^t)\|, \|x^{t+1}-z^t\|\}\le\delta,$$
then
$$\|x(y^t_+(z^t), z^t)-x^*(z^t)\|\le \sigma_5\|y^t-y^t_+(z^t)\|.$$
holds for  some constant $\sigma_5>0$.
\end{lemma}

Using this lemma, we can prove Theorem \ref{discrete}.
We choose $\beta$ sufficiently small, then when the residuals appear in \eqref{basic estimate} are large, we can prove that $\phi^t$ decreases sufficiently using the compactness of $Y$. When the residuals are  small, the error bound Lemma~\ref{dual error bound} can be used to guarantee the sufficient decrease of $\phi^t$.
Therefore, \eqref{suff-decrease} always holds, which yields  Theorem~\ref{discrete}. However, for the general nonconvex-concave problem \ref{minimax1}, we can only have a ``weaker'' bound.
\begin{lemma}\label{weak error bound}
Suppose Assumption \ref{basic-ass} holds for problem \ref{minimax1}.
Define $D(Y)$ to be the diameter of $Y$. If Assumption \ref{basic-ass} holds, we have
\begin{eqnarray*}
\alpha (p-L)\|x^*(z^t)-x(y^t_+(z^t), z^t)\|^2 \le (1+\alpha L + \alpha L \sigma_2)\|y^t-y^t_+(z^t)\|\cdot D(Y)
\end{eqnarray*}
for some $\sigma_2 > 0$.
\end{lemma}

Note that this is a nonhomogeneous error bound, which can help us bound the term $\|x^*(z^t)-x(y^t_+(z^t), z^t)\|$ only when $\|y^t-y^t_+(z^t)\|$ is not too small. Therefore, we say it is ``weaker'' dual error bound.
To obtain an $\epsilon$-stationary solution,  we need to choose $\beta$ sufficiently small and proportional to $\epsilon^2$.
In this case, we can prove that if $\phi^t$ stops to decrease, we have already obtained an $\epsilon$-stationary solution by Lemma~\ref{weak error bound}. By the remark after \eqref{suff-decrease}, we need $\mathcal{O}((1/(\epsilon\sqrt{\beta}))^2)=\mathcal{O}(1/\epsilon^4)$ iterations to obtain an $\epsilon$-stationary solution.

\vspace{3mm}

\noindent{\bf Remark.} For the general nonconvex-concave problem \eqref{minimax1}, we need to choose $\beta$ proportional to $\epsilon^2$ and hence the iteration complexity is higher than the previous case. However, it is expected that for a concrete problem with some special structure, the ``weaker" error bound Lemma~\ref{weak error bound} can be improved, as is the iteration complexity bound. This is left as a future work.

\vspace{3mm}

The proof sketch can be summarized in the following steps:
\begin{itemize}
    \item In Step 1, we introduce the potential function $\phi^t$ which is shown to be bounded below.
    To obtain the convergence rate of the algorithms, we want to prove the potential function can make sufficient decrease at every iterate $t$, i.e., we want to show $\phi^t - \phi^{t+1} > 0$. 
    \item In Step 2, we study this difference $\phi^t - \phi^{t+1}$ and provide a lower bound of it in Proposition \ref{basic estimate}. Notice that a negative term (\ref{eqn: negative term}) will show up in the lower bound, and we have to carefully analyze the magnitude of this term to obtain  $\phi^t - \phi^{t+1} > 0$. 
    \item Analyzing the negative term is the main difficulty of the proof. In Step 3, we discuss how to deal with this difficulty for solving Problem \ref{minimax1} and Problem \ref{minmax2} separately.
    \item Finally, we show the potential function makes a sufficient decrease at every iterate as stated in (\ref{suff-decrease}), and conclude our proof by computing the number of iterations to achieve an $\epsilon$-solution (as shown in Lemma \ref{trivial-appe} the appendix).
\end{itemize}

\section{Numerical Results on Robust Neural Network Training}\label{sec: experiment}
In this section, we apply the Smoothed-GDA algorithm to train a robust neural network on MNIST data set against adversarial attacks \cite{goodfellow2014explaining, kurakin2016adversarial, madry2017towards}.
The optimization formulation is
\begin{align}
\label{eq: Madry2}
\min_{\mathbf{w}}\, \; \sum_{i=1}^{N}\;\max_{\delta_i,\;\text{s.t.}\;|\delta_i|_{\infty}\leq \varepsilon}   {\ell}(f(x_i+\delta_i;\mathbf{w}), y_i),
\end{align}
where $\mathbf{w}$ is the parameter of the neural network, the pair $(x_i,y_i)$ denotes the $i$-th data point,
and $\delta_i$ is the perturbation added to data point~$i$.
As \eqref{eq: Madry2} is difficult to solve directly, researchers \cite{nouiehed2019solving} have proposed an approximation of \eqref{eq: Madry2}  as the following nonconvex-concave problem, which is in the form of \eqref{minmax2} we discussed before.
\begin{align}
\label{eq:finite-max2}
\min_{\mathbf{w}}\;\sum_{i=1}^{N}\;\max_{\mathbf{t} \in \mathcal{T}}\sum_{j=0}^{9}\;{t_j}\ell\left(f\left(x_{ij}^K;\mathbf{w}\right),y_i\right),\; \mathcal{T} =\left\{(t_1, \cdots, t_m)\mid \sum_{i=1}^mt_i=1, t_i\ge 0\right\},
\end{align}
where $K$ is a parameter in the approximation, and $x_{ij}^K$ is an approximated attack on sample $x_i$ by changing the output of the network to label $j$. The details of this formulation and the structure of the network in experiments are provided in the appendix.
\begin{table}[H]
\centering
\resizebox{\textwidth}{!}{%
\begin{tabular}{@{}cccccccc@{}}
\toprule
\multicolumn{1}{c}{} & {Natural} & \multicolumn{3}{c}{$\text{FGSM}\;L_{\infty}$~\cite{goodfellow2014explaining}} & \multicolumn{3}{c}{$\text{PGD}^{40}\; L_{\infty}$~\cite{kurakin2016adversarial}} \\ \cmidrule(l){3-8}
\multicolumn{1}{c}{} &  & $\varepsilon=0.2$ & $\varepsilon=0.3$ & $\varepsilon=0.4$ & $\varepsilon=0.2$ & $\varepsilon=0.3$ & $\varepsilon=0.4$ \\ \midrule
\cite{madry2017towards} with $\varepsilon=0.35$ & 98.58\% & 96.09\% & 94.82\% & 89.84\% & 94.64\% & 91.41\% & 78.67\% \\
\cite{zhang2019theoretically} with $\varepsilon=0.35$ & 97.37\% & 95.47\% & 94.86\% & 79.04\% & 94.41\% & 92.69\% & 85.74\% \\
\cite{zhang2019theoretically} with $\varepsilon=0.40$ & 97.21\% & 96.19\% & 96.17\% & 96.14\% & 95.01\% & 94.36\% & 94.11\% \\
\cite{nouiehed2019solving} with $\varepsilon=0.40$ & 98.20\% & 97.04\% & 96.66\% & \bf{96.23}\% & 96.00\% & 95.17 \% & 94.22\% \\
Smoothed-GDA with $\varepsilon=0.40$ & \bf{98.89}\% & \bf{97.87}\% & \bf{97.23}\% & 95.81\% & \bf{96.71}\% & \bf{95.62}\% & \bf{94.51}\% \\\bottomrule
\end{tabular}%
}
\vskip 0.1in
\caption{Test accuracies under FGSM and PGD attacks.
}
\label{tab:robust_nn_results}
\end{table}
\begin{figure}
    \centering
    \includegraphics[width=0.5\textwidth]{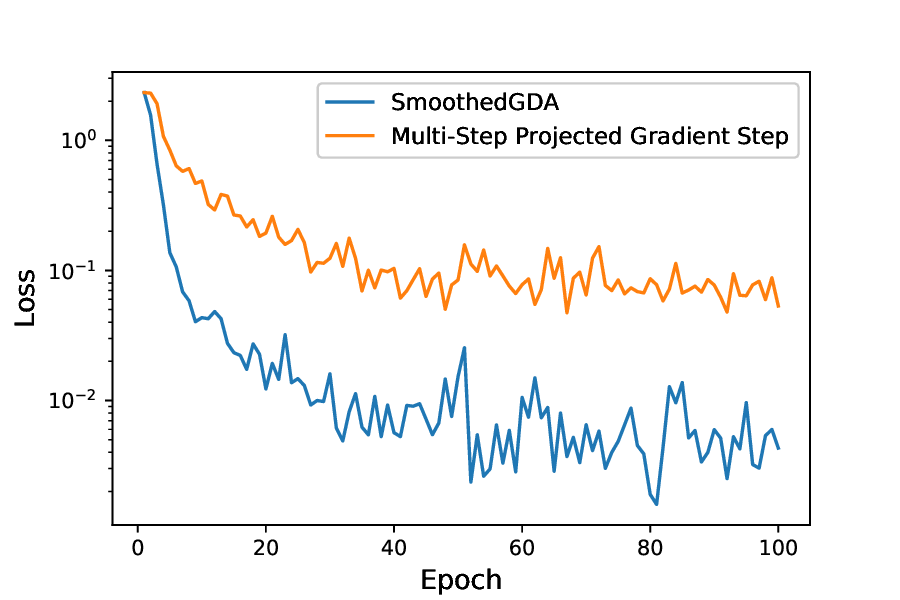}
    \caption{Convergence speed of Smoothed-GDA and the algorithm in \cite{nouiehed2019solving}.
    }
    \label{fig: test loss}
\end{figure}
\noindent\textbf{Results:} We compare our results with three algorithms from \cite{madry2017towards, zhang2019theoretically, nouiehed2019solving}.
The references \cite{madry2017towards, zhang2019theoretically} are two classical algorithms in adversarial training,
while the recent reference \cite{nouiehed2019solving} considers the same problem formulation as \eqref{minmax2} and has an algorithm with $\mathcal{O}(1/\epsilon^{3.5})$ iteration complexity. The accuracy of our formulation are summarized in Table~\ref{tab:robust_nn_results} which shows that the formulation \eqref{minmax2} leads to a comparable or slightly better performance to the other algorithms.
We also compare the convergence on the loss function when using the Smoothed-GDA algorithm and the one in \cite{nouiehed2019solving}. In Figure~\ref{fig: test loss}, Smoothed-GDA algorithm takes only 5 epochs to get the loss values below 0.2 while the algorithm proposed in \cite{nouiehed2019solving} takes more than 14 epochs. In addition, the loss obtained from the Smoothed-GDA algorithm has a smaller variance.

\section{Conclusion}
In this paper, we propose a simple single-loop algorithm for nonconvex min-max problems \eqref{minimax1}. For an important family of problems \eqref{minmax2}, the algorithm is even more efficient due to the  {\bf dual error bound},  and it is well-suited for problems in large-size dimensions and distributed setting.  The algorithmic framework is flexible, and hence in the future work, we can extend the algorithm to more practical problems and derive {\bf stronger} error bounds to attain lower iteration complexity.

\newpage
\section*{Broader Impact}
In this paper, we propose a single-loop algorithm for min-max problem. This algorithm is easy to implemented and proved to be efficient in a family of nonconvex minimax problems and have good numerical behavior in robust training.
This paper focuses on theoretical study of the algorithms. In industrial applications, several aspects of impact can be expected:

\begin{enumerate}
\item \textbf{Save energy by improving efficiency.} The trick developed in this paper has the potential to accelerate the training for machine learning problems involving a minimax problem such robust training for uncertain data, generative adversarial net(GAN) and AI for games. This means that the actual training time will decrease dramatically by using our algorithm. Training neural network is very energy-consuming, and reducing the training time can help the industries or companies to save energy.
\item \textbf{Promote fairness.} We consider min-max problems in this paper. A model that is trained under this framework will not allow poor performance on some objectives in order to boost performance on the others. Therefore, even if the training data itself is biased, the model will not allow some objectives to contribute heavily to minimizing the average loss due to the min-max framework. In other words, this framework promotes fairness, and model that is trained under this framework will provide fair solutions to the problems.
\item \textbf{Provide flexible framework.} Our algorithmic framework is flexible. Though in the paper, we only discuss some general formulation, our algorithm can be easily extended to many practical settings. For example, based on our general framework for multi-block problems, we can design algorithms efficiently solving problems with distributedly stored data, decentralized control or privacy concern. Therefore, our algorithm may have an impact on some popular big data applications such as distributed training, federated learning  and so on.
\end{enumerate}

\section*{Funding Disclosure}
This research is supported by 
the leading talents of Guangdong Province program [Grant 00201501]; 
the National Science Foundation of China [Grant 61731018];
the Air Force Office of Scientific Research [Grant FA9550-12-1-0396]; 
the National Science Foundation [Grant CCF 1755847];
Shenzhen Peacock Plan [Grant KQTD2015033114415450];
the Development and Reform Commission of Shenzhen Municipality;
and Shenzhen Research Institute of Big Data. 


\newpage
\bibliography{reference}

\begin{thebibliography}{10}

\bibitem{goodfellow2014generative}
I.~Goodfellow, J.~Pouget-Abadie, M.~Mirza, B.~Xu, D.~Warde-Farley, S.~Ozair,
  A.~Courville, and Y.~Bengio, ``Generative adversarial nets,'' in {\em
  Advances in neural information processing systems}, pp.~2672--2680, 2014.

\bibitem{arjovsky2017wasserstein}
M.~Arjovsky, S.~Chintala, and L.~Bottou, ``Wasserstein gan,'' {\em arXiv
  preprint arXiv:1701.07875}, 2017.

\bibitem{madry2017towards}
A.~Madry, A.~Makelov, L.~Schmidt, D.~Tsipras, and A.~Vladu, ``Towards deep
  learning models resistant to adversarial attacks,'' {\em arXiv preprint
  arXiv:1706.06083}, 2017.

\bibitem{ben2009robust}
A.~Ben-Tal, L.~El~Ghaoui, and A.~Nemirovski, {\em Robust optimization},
  vol.~28.
\newblock Princeton University Press, 2009.

\bibitem{delage2010distributionally}
E.~Delage and Y.~Ye, ``Distributionally robust optimization under moment
  uncertainty with application to data-driven problems,'' {\em Operations
  research}, vol.~58, no.~3, pp.~595--612, 2010.

\bibitem{namkoong2016stochastic}
H.~Namkoong and J.~C. Duchi, ``Stochastic gradient methods for distributionally
  robust optimization with f-divergences,'' in {\em Advances in neural
  information processing systems}, pp.~2208--2216, 2016.

\bibitem{namkoong2017variance}
H.~Namkoong and J.~C. Duchi, ``Variance-based regularization with convex
  objectives,'' in {\em Advances in neural information processing systems},
  pp.~2971--2980, 2017.

\bibitem{zhang2017stochastic}
Y.~Zhang and L.~Xiao, ``Stochastic primal-dual coordinate method for
  regularized empirical risk minimization,'' {\em The Journal of Machine
  Learning Research}, vol.~18, no.~1, pp.~2939--2980, 2017.

\bibitem{tan2018stochastic}
C.~Tan, T.~Zhang, S.~Ma, and J.~Liu, ``Stochastic primal-dual method for
  empirical risk minimization with o (1) per-iteration complexity,'' in {\em
  Advances in Neural Information Processing Systems}, pp.~8366--8375, 2018.

\bibitem{du2017stochastic}
S.~S. Du, J.~Chen, L.~Li, L.~Xiao, and D.~Zhou, ``Stochastic variance reduction
  methods for policy evaluation,'' in {\em Proceedings of the 34th
  International Conference on Machine Learning-Volume 70}, pp.~1049--1058,
  JMLR. org, 2017.

\bibitem{dai2017sbeed}
B.~Dai, A.~Shaw, L.~Li, L.~Xiao, N.~He, Z.~Liu, J.~Chen, and L.~Song, ``Sbeed:
  Convergent reinforcement learning with nonlinear function approximation,''
  {\em arXiv preprint arXiv:1712.10285}, 2017.

\bibitem{nemirovski2004prox}
A.~Nemirovski, ``Prox-method with rate of convergence o (1/t) for variational
  inequalities with lipschitz continuous monotone operators and smooth
  convex-concave saddle point problems,'' {\em SIAM Journal on Optimization},
  vol.~15, no.~1, pp.~229--251, 2004.

\bibitem{nesterov2007dual}
Y.~Nesterov, ``Dual extrapolation and its applications to solving variational
  inequalities and related problems,'' {\em Mathematical Programming},
  vol.~109, no.~2-3, pp.~319--344, 2007.

\bibitem{monteiro2010complexity}
R.~D. Monteiro and B.~F. Svaiter, ``On the complexity of the hybrid proximal
  extragradient method for the iterates and the ergodic mean,'' {\em SIAM
  Journal on Optimization}, vol.~20, no.~6, pp.~2755--2787, 2010.

\bibitem{palaniappan2016stochastic}
B.~Palaniappan and F.~Bach, ``Stochastic variance reduction methods for
  saddle-point problems,'' in {\em Advances in Neural Information Processing
  Systems}, pp.~1416--1424, 2016.

\bibitem{gidel2016frank}
G.~Gidel, T.~Jebara, and S.~Lacoste-Julien, ``Frank-wolfe algorithms for saddle
  point problems,'' {\em arXiv preprint arXiv:1610.07797}, 2016.

\bibitem{mertikopoulos2018optimistic}
P.~Mertikopoulos, B.~Lecouat, H.~Zenati, C.-S. Foo, V.~Chandrasekhar, and
  G.~Piliouras, ``Optimistic mirror descent in saddle-point problems: Going the
  extra (gradient) mile,'' {\em arXiv preprint arXiv:1807.02629}, 2018.

\bibitem{hamedani2018iteration}
E.~Y. Hamedani, A.~Jalilzadeh, N.~Aybat, and U.~Shanbhag, ``Iteration
  complexity of randomized primal-dual methods for convex-concave saddle point
  problems,'' {\em arXiv preprint arXiv:1806.04118}, 2018.

\bibitem{mokhtari2019unified}
A.~Mokhtari, A.~Ozdaglar, and S.~Pattathil, ``A unified analysis of
  extra-gradient and optimistic gradient methods for saddle point problems:
  Proximal point approach,'' {\em arXiv preprint arXiv:1901.08511}, 2019.

\bibitem{nouiehed2019solving}
M.~Nouiehed, M.~Sanjabi, T.~Huang, J.~D. Lee, and M.~Razaviyayn, ``Solving a
  class of non-convex min-max games using iterative first order methods,'' in
  {\em Advances in Neural Information Processing Systems}, pp.~14905--14916,
  2019.

\bibitem{collins2020distribution}
L.~Collins, A.~Mokhtari, and S.~Shakkottai, ``Distribution-agnostic
  model-agnostic meta-learning,'' {\em arXiv preprint arXiv:2002.04766}, 2020.

\bibitem{ostrovskii2020efficient}
D.~M. Ostrovskii, A.~Lowy, and M.~Razaviyayn, ``Efficient search of first-order
  nash equilibria in nonconvex-concave smooth min-max problems,'' {\em arXiv
  preprint arXiv:2002.07919}, 2020.

\bibitem{lin2020near}
T.~Lin, C.~Jin, M.~Jordan, {\em et~al.}, ``Near-optimal algorithms for minimax
  optimization,'' {\em arXiv preprint arXiv:2002.02417}, 2020.

\bibitem{lu2019hybrid}
S.~Lu, I.~Tsaknakis, M.~Hong, and Y.~Chen, ``Hybrid block successive
  approximation for one-sided non-convex min-max problems: algorithms and
  applications,'' {\em arXiv preprint arXiv:1902.08294}, 2019.

\bibitem{thekumparampil2019efficient}
K.~K. Thekumparampil, P.~Jain, P.~Netrapalli, and S.~Oh, ``Efficient algorithms
  for smooth minimax optimization,'' in {\em Advances in Neural Information
  Processing Systems}, pp.~12659--12670, 2019.

\bibitem{jin2019minmax}
C.~Jin, P.~Netrapalli, and M.~I. Jordan, ``Minmax optimization: Stable limit
  points of gradient descent ascent are locally optimal,'' {\em arXiv preprint
  arXiv:1902.00618}, 2019.

\bibitem{rafique2018non}
H.~Rafique, M.~Liu, Q.~Lin, and T.~Yang, ``Non-convex min-max optimization:
  Provable algorithms and applications in machine learning,'' {\em arXiv
  preprint arXiv:1810.02060}, 2018.

\bibitem{lin2019gradient}
T.~Lin, C.~Jin, and M.~I. Jordan, ``On gradient descent ascent for
  nonconvex-concave minimax problems,'' {\em arXiv preprint arXiv:1906.00331},
  2019.

\bibitem{xu2020unified}
Z.~Xu, H.~Zhang, Y.~Xu, and G.~Lan, ``A unified single-loop alternating
  gradient projection algorithm for nonconvex-concave and convex-nonconcave
  minimax problems,'' {\em arXiv preprint arXiv:2006.02032}, 2020.

\bibitem{zhang2020proximal}
J.~Zhang and Z.-Q. Luo, ``A proximal alternating direction method of multiplier
  for linearly constrained nonconvex minimization,'' {\em SIAM Journal on
  Optimization}, vol.~30, no.~3, pp.~2272--2302, 2020.

\bibitem{kurakin2016adversarial}
A.~Kurakin, I.~Goodfellow, and S.~Bengio, ``Adversarial machine learning at
  scale,'' {\em arXiv preprint arXiv:1611.01236}, 2016.

\bibitem{goodfellow2014explaining}
I.~J. Goodfellow, J.~Shlens, and C.~Szegedy, ``Explaining and harnessing
  adversarial examples,'' {\em arXiv preprint arXiv:1412.6572}, 2014.

\bibitem{hardt2016equality}
M.~Hardt, E.~Price, and N.~Srebro, ``Equality of opportunity in supervised
  learning,'' in {\em Advances in neural information processing systems},
  pp.~3315--3323, 2016.

\bibitem{dwork2012fairness}
C.~Dwork, M.~Hardt, T.~Pitassi, O.~Reingold, and R.~Zemel, ``Fairness through
  awareness,'' in {\em Proceedings of the 3rd innovations in theoretical
  computer science conference}, pp.~214--226, 2012.

\bibitem{mohri2019agnostic}
M.~Mohri, G.~Sivek, and A.~T. Suresh, ``Agnostic federated learning,'' {\em
  arXiv preprint arXiv:1902.00146}, 2019.

\bibitem{finn2017model}
C.~Finn, P.~Abbeel, and S.~Levine, ``Model-agnostic meta-learning for fast
  adaptation of deep networks,'' in {\em Proceedings of the 34th International
  Conference on Machine Learning-Volume 70}, pp.~1126--1135, JMLR. org, 2017.

\bibitem{parikh2014proximal}
N.~Parikh and S.~Boyd, ``Proximal algorithms,'' {\em Foundations and Trends in
  optimization}, vol.~1, no.~3, pp.~127--239, 2014.

\bibitem{nesterov2005smooth}
Y.~Nesterov, ``Smooth minimization of non-smooth functions,'' {\em Mathematical
  programming}, vol.~103, no.~1, pp.~127--152, 2005.

\bibitem{harker1990finite}
P.~T. Harker and J.-S. Pang, ``Finite-dimensional variational inequality and
  nonlinear complementarity problems: a survey of theory, algorithms and
  applications,'' {\em Mathematical programming}, vol.~48, no.~1-3,
  pp.~161--220, 1990.

\bibitem{facchinei2007finite}
F.~Facchinei and J.-S. Pang, {\em Finite-dimensional variational inequalities
  and complementarity problems}.
\newblock Springer Science \& Business Media, 2007.

\bibitem{forsgren2002interior}
A.~Forsgren, P.~E. Gill, and M.~H. Wright, ``Interior methods for nonlinear
  optimization,'' {\em SIAM review}, vol.~44, no.~4, pp.~525--597, 2002.

\bibitem{carbonetto2009interior}
P.~Carbonetto, M.~Schmidt, and N.~D. Freitas, ``An interior-point stochastic
  approximation method and an l1-regularized delta rule,'' in {\em Advances in
  neural information processing systems}, pp.~233--240, 2009.

\bibitem{liang2014local}
J.~Liang, J.~Fadili, and G.~Peyr{\'e}, ``Local linear convergence of
  forward--backward under partial smoothness,'' in {\em Advances in Neural
  Information Processing Systems}, pp.~1970--1978, 2014.

\bibitem{lu2019snap}
S.~Lu, M.~Razaviyayn, B.~Yang, K.~Huang, and M.~Hong, ``Snap: Finding
  approximate second-order stationary solutions efficiently for non-convex
  linearly constrained problems,'' {\em arXiv preprint arXiv:1907.04450}, 2019.

\bibitem{zeng2018nonconvex}
J.~Zeng and W.~Yin, ``On nonconvex decentralized gradient descent,'' {\em IEEE
  Transactions on signal processing}, vol.~66, no.~11, pp.~2834--2848, 2018.

\bibitem{cannelli2019asynchronous}
L.~Cannelli, F.~Facchinei, V.~Kungurtsev, and G.~Scutari, ``Asynchronous
  parallel algorithms for nonconvex optimization,'' {\em Mathematical
  Programming}, pp.~1--34, 2019.

\bibitem{wang2019global}
Y.~Wang, W.~Yin, and J.~Zeng, ``Global convergence of admm in nonconvex
  nonsmooth optimization,'' {\em Journal of Scientific Computing}, vol.~78,
  no.~1, pp.~29--63, 2019.

\bibitem{xu2015block}
Y.~Xu and W.~Yin, ``Block stochastic gradient iteration for convex and
  nonconvex optimization,'' {\em SIAM Journal on Optimization}, vol.~25, no.~3,
  pp.~1686--1716, 2015.

\bibitem{defossez2020convergence}
A.~D{\'e}fossez, L.~Bottou, F.~Bach, and N.~Usunier, ``On the convergence of
  adam and adagrad,'' {\em arXiv preprint arXiv:2003.02395}, 2020.

\bibitem{xie2020maximum}
X.~Xie, H.~Kong, J.~Wu, W.~Zhang, G.~Liu, and Z.~Lin,
  ``Maximum-and-concatenation networks,'' {\em arXiv preprint
  arXiv:2007.04630}, 2020.

\bibitem{zhang2019theoretically}
H.~Zhang, Y.~Yu, J.~Jiao, E.~P. Xing, L.~E. Ghaoui, and M.~I. Jordan,
  ``Theoretically principled trade-off between robustness and accuracy,'' {\em
  arXiv preprint arXiv:1901.08573}, 2019.

\bibitem{bertsekas2009convex}
D.~P. Bertsekas, {\em Convex optimization theory}.

\bibitem{berger2012differential}
M.~Berger and B.~Gostiaux, {\em Differential Geometry: Manifolds, Curves, and
  Surfaces: Manifolds, Curves, and Surfaces}, vol.~115.
\newblock Springer Science \& Business Media, 2012.

\bibitem{wang2013projection}
W.~Wang and M.~A. Carreira-Perpin{\'a}n, ``Projection onto the probability
  simplex: An efficient algorithm with a simple proof, and an application,''
  {\em arXiv preprint arXiv:1309.1541}, 2013.

\bibitem{krizhevsky2009learning}
A.~Krizhevsky, G.~Hinton, {\em et~al.}, ``Learning multiple layers of features
  from tiny images,'' 2009.

\bibitem{wong2018scaling}
E.~Wong, F.~Schmidt, J.~H. Metzen, and J.~Z. Kolter, ``Scaling provable
  adversarial defenses,'' in {\em Advances in Neural Information Processing
  Systems}, pp.~8400--8409, 2018.

\end{thebibliography}
\bibliographystyle{ieeetr}

\newpage
\appendix

In the appendix, we will give the proof of the main theorems. The appendix is organized as follows:
\begin{enumerate}
\item In section \ref{Appendix:Notations}, we list some notations used in the appendix.
\item In section \ref{Appendix: one block}, we prove the two main theorems in one-block case.
\item In section \ref{Appendix: multi block}, we briefly state the proof of the two main theorems in multi-block setting.
\item In section \ref{Appendix: error bound}, we prove the two main error bound  lemmas (Lemma \ref{weak error bound} and Lemma \ref{dual error bound}) under the strict complementarity assumption.
\item In section \ref{Appendix: sc}, we see that the strict complementarity assumption can be relaxed to a weaker regularity assumption. We also prove that this weaker regularity assumption is generic for robust regression problems with square loss, i.e.,  we prove that our regularity assumption holds with probability $1$ if the data points are joint from a continuous distribution.
\item In the last section \ref{Appendix: experiment}, we give some more details about the experiment.
\end{enumerate}
\section{Notations}\label{Appendix:Notations}

We first list some notations which will be used  in the appendix.
\begin{enumerate}
\item $[m]=\{1, 2, \cdots, m\}$.
\item $W^*$ is the solution set of  \eqref{minimax1} or \eqref{minmax2}. $X^*$ is the set of all solutions $x^*$, i.e., $x^*\in X^*$ if there exists a $y^*$ such that $(x^*, y^*)\in W^*$.
\item $\mathcal{B}(r)$ is a Euclidian ball of radius $r$ for proper dimension.
 \item $\mathrm{dist}(v, S)$ means the Euclidian distance from a point $v$ to a set $S$.
\item For a vector $v$,  $v_i$ means the $i$-th component of $v$. For a set $\mathcal{S}$, $v_{\mathcal{S}}\in \mathbb{R}^{|\mathcal{S}|}$ is the vector containing all components  $v_i$'s with $i\in\mathcal{S}$.
 \item Let $\mathbf{A}\in \mathbb{R}^{n\times m}$ be a matrix and $\mathcal{S}\subseteq [m]$ be an index set. Then $\mathbf{A}_{\mathcal{S}}$
 represents the row sub-matrix of $A$ corresponding to the rows with index in $\mathcal{S}$.
\item For a matrix $\mathbf{M}$, $\gamma(\mathbf{M})$ is the smallest singular value of $\mathbf{M}$.
 \item The projection of a point $y$, onto a set $X$ is defined as $P_X(y) = \argmin_{x \in X} \frac{1}{2} \|x - y\|^2$.
\end{enumerate}

\section{Proof of the two main theorems: one-block case}\label{Appendix: one block}
In this section, we prove the two main theorems in one-block case. The proof of the multi-block case is similar and will be given in the next section.

\textbf{Proof Sketch.}

\begin{itemize}
    \item In Step 1, we will introduce the potential function $\phi^t$ which is shown to be bounded below.
    To obtain the convergence rate of the algorithms, we want to prove the potential function can make sufficient decrease at every iterate $t$, i.e., we want to show $\phi^t - \phi^{t+1} > 0$. 
    \item In Step 2, we will study this difference $\phi^t - \phi^{t+1}$ and provide a lower bound of it in Proposition \ref{basic estimate}. Notice that a negative term (\ref{eqn: negative term}) will show up in the lower bound, and we have to carefully analyze the magnitude of this term to obtain  $\phi^t - \phi^{t+1} > 0$. 
    \item Analyzing the negative term will be the main difficulty of the proof. In Step 3, we will discuss how to deal with this difficulty for solving Problem \ref{minimax1} and Problem \ref{minmax2} separately.
    \item Finally, we will show the potential function makes a sufficient decrease at every iterate as stated in (\ref{suff-decrease}), and will conclude our proof by computing the number of iterations to achieve an $\epsilon$-solution in Lemma \ref{trivial-appe}.
\end{itemize}

\subsection{The potential function and basic estimate}
Recall that the potential function is:
$$\phi^t=\Phi(x^t, z^t;y^t)=K(x^t, z^t;y^t)-2d(y^t, z^t)+2P(z^t),$$
where
\begin{eqnarray*}
K(x, z; y)&=&f(x, y)+\frac{p}{2}\|x-z\|^2,\\
d(y, z)&=&\min_{x\in X}K(x, z; y),\\
P(z)&=&\min_{x\in X}\max_{y\in Y}K(x, z; y).
\end{eqnarray*}
Also note that if $p>L$, $K(x, z; y)$ is strongly convex of $x$ with modular $p-L$ and $\nabla_xK(x, z; y)$ is Lipschitz-continuous of $x$ with a constant $L+p$.
We also use the following notations:
\begin{enumerate}
\item $h(x, z)=\max_{y\in Y}K(x, z; y)$.
\item $x(y, z)=\arg\min_{x\in X}K(x, z; y)$, $x^*(z)=\arg\min_{x\in X}h(x, z)$.
\item The set $Y(z)=\arg\max_{y\in Y}d(y, z)$.
\item $y_+(z)=P_Y(y+\alpha\nabla_yK(x(y, z), z; y))$.
\item $x_+(y, z)=P_X(x-c\nabla_xK(x, z; y))$.
\end{enumerate}
First of all, we can prove that $\phi^t$ is bounded from below:
\begin{lemma}\label{lower-bounded}
We have
$$\phi(x, y, z)\ge \underline{f}.$$

\end{lemma}
\begin{proof}
By the definition of $d(\cdot)$ and $P(\cdot)$, we have
\begin{equation}\label{weak-duality}
K(x, z; y)\ge d(y, z), \quad P(z)\ge d(y, z), \quad, P(z)\ge \underline{f}.
\end{equation}
Hence, we have
\begin{eqnarray*}
\phi(x, y, z)&=&P(z)+(K(x, z; y)-d(y, z))+(P(z)-d(y, z))\\
&\ge&P(z)\\
&\ge\underline{f}.
\end{eqnarray*}
\end{proof}
Next, we state some ``error bounds''.
\begin{lemma}\label{Lp_continue}
There exist constants $\sigma_1, \sigma_2, \sigma_3$ independent of $y$ such that
\begin{eqnarray}
 &&\|x(y, z)-x(y, z')\|\le \sigma_1\|z-z'\|, \label{eb1}\\
&&\|x^*(z)-x^*(z')\|\le \sigma_1\|z-z'\|, \label{eb11}\\
&&\|x(y, z)-x(y', z)\|\le \sigma_2\|y-y'\|,\label{eb2}\\
&&\|x^{t+1}-x(y^t, z^t)\|\le \sigma_3\|x^t-x^{t+1}\|,\label{eb3}
\end{eqnarray}
for any $y, y'\in Y$ and $z, z'\in X$,
where $\sigma_1=\frac{p}{-L+p}$, $\sigma_2=\frac{2(p+L)}{p-L}$, $\sigma_3=\frac{1+(c(-L+p))}{c(-L+p)}$..
\end{lemma}
\begin{proof}
The proofs of \eqref{eb1}, \eqref{eb11}  and \eqref{eb3} are the same as those in   Lemma 3.10 in \cite{zhang2020proximal} and hence omitted. 
We only need to prove \eqref{eb2}.
Using the strong convexity of $K(\cdot, z; y)$ of $x$, we have
\begin{eqnarray}\label{str}
&&K(x(y, z), z; y)-K(x(y', z), z; y)\le-\frac{-L+p}{2}\|x(y, z)-x(y', z)\|^2,\\
&&K(x(y, z), z; y')-K(x(y', z), z; y')\ge\frac{-L+p}{2}\|x(y, z)-x(y', z)\|^2.
\end{eqnarray}
Moreover, using the concavity of $K(x, z;\cdot)$ of $y$, we have
\begin{eqnarray}\label{concave}
\nonumber &&K(x(y, z), z; y')-K(x(y, z), z; y)\\
&\le& \langle \nabla_yK(x(y, z), z; y), y'-y\rangle.
\end{eqnarray}
Using the Lipschitz-continuity of $\nabla_yK(x, z;\cdot)$, we have
\begin{eqnarray}\label{lip}
\nonumber &&K(x(y', z), z; y)-K(x(y', z), z; y')\\
&\le&\langle \nabla_yK(x(y', z), z; y), y'-y\rangle+\frac{L}{2}\|y-y'\|^2.
\end{eqnarray}
Combining \eqref{str} and \eqref{lip}, we have

\begin{eqnarray}\label{quadratic-inequality}
&&(-L+p)\|x(y, z)-x(y', z)\|^2\\
&\le&\langle \nabla_yK(x(y, z), z; y)-\nabla_yK(x(y', z), z; y), y'-y\rangle+\frac{L}{2}\|y-y'\|^2\\
&\le&(p+L)\|x(y', z)-x(y, z)\|\|y-y'\|+\frac{L}{2}\|y-y'\|^2,
\end{eqnarray}
where the last inequality uses the Cauchy-schwarz inequality and the Lipschitz-continuity of $\nabla_xK(\cdot, z; y)$ of $x$.

Let $\zeta=\|x(y, z)-x(y', z)\|/\|y-y'\|$.
Then \eqref{quadratic-inequality} becomes
$$\zeta^2\le \frac{p+L}{p-L}\zeta+\frac{L}{2(p-L)}.$$
Hence, we only need to solve the above quadratic inequality.
We have
\begin{eqnarray*}
\zeta^2&\le&\frac{1}{2}\zeta^2+\frac{1}{2}\left(\frac{p+L}{p-L}\right)^2+\frac{L}{2(p-L)}\\
&\le&\frac{1}{2}\zeta^2+\frac{1}{2}\left(\frac{p+L}{p-L}\right)^2+\frac{p+L}{2(p-L)}\\
&\le&\frac{1}{2}\zeta^2+\frac{1}{2}\left(\frac{p+L}{p-L}\right)^2+\frac{1}{2}\left(\frac{p+L}{p-L}\right)^2\\
&=&\frac{1}{2}\zeta^2+\left(\frac{p+L}{p-L}\right)^2,
\end{eqnarray*}
where the first inequality is due to the AM-GM inequality and the third inequality is because $(p+L)/(p-L)>1$.
Therefore
$$\zeta\le \sqrt{2}\frac{p+L}{p-L}<2\frac{p+L}{p-L}.$$
Hence, we can take $\sigma_2=2(p+L)/(p-L)$ and finish the proof.
\end{proof}
The following lemma is a direct corollary of the above lemma:
\begin{lemma}\label{Lipschitz-dual}
The dual function $d(\cdot, z)$ is a differentiable function of $ y$with Lipschitz continuous gradient
$$\nabla_yd(y, z)=\nabla_yK(x(y, z), z; y)=\nabla_yf(x(y, z), y)$$
and
\[
\|\nabla_y d(y',z)-\nabla_y d(y'',z)\|\le {L}_d\|y'-y''\|,\quad  \forall\; y', y''\in Y.
\]
with $L_d=L+L\sigma_2$.
\end{lemma}
\noindent{\bf Remark.}Note that if $p\ge 3L$, then we have $\sigma_2=2(p+L)/(p-L)\ge 4L$ and hence
\begin{equation}\label{5L}
L_d\ge5L.
\end{equation}
\begin{proof}
Using Danskin's theorem in convex analysis \cite{bertsekas2009convex}, we know that $d$ is a differentiable function with
$$\nabla_yd(y, z)=\nabla_yK(x(y, z), z; y)=\nabla_yf(x(y, z), y).$$
To prove the Lipschitz-continuity, we have
\begin{eqnarray*}
\|\nabla_y d(y',z)-\nabla_y d(y'', z)\|&=&\|\nabla_y K(x(y',z),z;y')-\nabla_y K(x(y'',z),z;y'')\|\\
&\le& \|\nabla_y K(x(y',z),z;y')-\nabla_y K(x(y',z),z;y'')\|\\
&&+\|\nabla_y K(x(y',z),z;y'')-\nabla_y K(x(y'',z),z;y'')\|\\
&\leq&L\|y'-y''\|+L\|x(y',z)-x(y'',z)\|\\
&\leq&L\|y'-y''\|+L\sigma_2\|y'-y''\|={L}_d\|y'-y''\|,
\end{eqnarray*}
where the last inequality is due to Lemma \ref{Lp_continue}.
\end{proof}

We then prove the following basic estimate.
\begin{proposition}\label{basic-estimate-appe}
We let
\begin{align}\label{eqn: parameters for Alg 2 general case-appe}
    p>3L, c<\frac{1}{p+L}, \alpha<\min\{\frac{1}{11L}, \frac{1}{4L\sigma_3^2}\}=\min\{\frac{1}{11L}, \frac{c^2(p-L)^2}{4L(1+c(p-L))^2}\}, 
    \beta<\min\{\frac{1}{36}, \frac{(p-L)^2}{384\alpha p(p+L)^2}\}.
\end{align}
Then we have
\begin{eqnarray}\label{basic estimate-appe}
&&\phi^t-\phi^{t+1}\\
\nonumber&\ge&\frac{1}{8c}\|x^t-x^{t+1}\|^2\\
&&+\frac{1}{8\alpha}\|y^t-y^t_+(z^t)\|^2+\frac{p}{8\beta}\|z^t-z^{t+1}\|^2\\
&&-24p\beta\|x^*(z^{t})-x(y^t_+(z^t), z^{t})\|^2
\end{eqnarray}
\end{proposition}
To prove this basic estimate, we need a series of lemmas.
\begin{lemma}[Primal Descent] \label{primal}
For any $t$, we have
\begin{eqnarray}\label{K0}
\nonumber K(x^t, z^t; y^t)-K(x^{t+1}, z^{t+1}; y^{t+1})&\ge& \frac{1}{2c}\|x^t-x^{t+1}\|^2+ \langle\nabla_y K(x^{t+1}, z^t; y^t), y^t - y^{t+1}\rangle\\
&&-\frac{L}{2}\|y^t-y^{t+1}\|^2+\frac{p}{2\beta}\|z^t-z^{t+1}\|^2.
\end{eqnarray}
\end{lemma}
\begin{proof}
Notice that the step of updating $x$ is a standard gradient projection, hence we have
\begin{eqnarray}\label{K1}
K(x^t, z^t; y^t)-K(x^{t+1}, z^t; y^t)\ge\frac{1}{2c}\|x^t-x^{t+1}\|^2.
\end{eqnarray}
Next, because $\nabla_yK(x, z;y)$ is $L$-Lipschitz-continuous of $y$, 
we have
\begin{eqnarray}\label{K2}
 K(x^{t+1}, z^t; y^t)-K(x^{t+1}, z^t; y^{t+1})\ge \langle \nabla_y K(x^{t+1}, z^t; y^t), y^t - y^{t+1}\rangle -\frac{L}{2}\|y^t-y^{t+1}\|^2.
\end{eqnarray}
Based on the update of variable $z^{t+1}$, i.e. $z^{t+1}=z^t+\beta(x^{t+1}-z^t)$, it is easy to show that
\begin{eqnarray}\label{K3}
\nonumber K(x^{t+1}, z^t; y^{t+1})-K(x^{t+1}, z^{t+1}; y^{t+1})\ge \frac{p}{2\beta}\|z^t-z^{t+1}\|^2.
\end{eqnarray}
Combining (\ref{K1})-(\ref{K3}), we finish the proof.
\end{proof}

\begin{lemma}[Dual Ascent]\label{dual ascent}
For any $t$, we have
\begin{eqnarray}\label{D0}
\nonumber d(y^{t+1}, z^{t+1})-d(y^t, z^t)
&\ge&  \langle \nabla_yd(y^t, z^t), y^{t+1}-y^t\rangle-\frac{L_d}{2}\|y^t-y^{t+1}\|^2\\
&&+\frac{p}{2}(z^{t+1}-z^t)^T(z^{t+1}+z^t-2x(y^{t+1}, z^{t+1}))\\
&=&  \langle \nabla_y K(x(y^t, z^t),z^t;y^t),y^{t+1}-y^t\rangle-\frac{L_d}{2}\|y^t-y^{t+1}\|^2\\
&&+\frac{p}{2}(z^{t+1}-z^t)^T(z^{t+1}+z^t-2x(y^{t+1}, z^{t+1}))
\end{eqnarray}
\end{lemma}
\begin{proof}
Using Lemma \ref{Lipschitz-dual}, we have
\begin{eqnarray}\label{D1}
\nonumber -d(y^{t+1},z^t)-(-d(y^t,z^t))&\leq& -\langle \nabla_y d(y^t,z^t),y^{t+1}-y^t\rangle+\frac{L_d}{2}\|y^t-y^{t+1}\|^2\\
&=&  \langle \nabla_y K(x(y^t, z^t),z^t;y^t),y^{t+1}-y^t\rangle-\frac{L_d}{2}\|y^t-y^{t+1}\|^2\\
\end{eqnarray}

Next,
\begin{eqnarray}\label{D2}
\nonumber &&d(y^{t+1},z^{t+1})-d(y^{t+1},z^t)\\
\nonumber &=&K(x(y^{t+1},z^{t+1}),z^{t+1};y^{t+1})-K(x(y^{t+1},z^{t}),z^{t};y^{t+1})\\
\nonumber &\geq&K(x(y^{t+1},z^{t+1}),z^{t+1};y^{t+1})-K(x(y^{t+1},z^{t+1}),z^{t};y^{t+1})\\
\nonumber &=& \frac{p}{2}\|x(y^{t+1},z^{t+1})-z^{t+1}\|^2-\frac{p}{2}\|x(y^{t+1},z^{t+1})-z^{t}\|^2\\
 &=& \frac{p}{2}(z^{t+1}-z^t)^T(z^{t+1}+z^t-2x(y^{t+1}, z^{t+1})).
\end{eqnarray}
Finally, using (\ref{D1})-(\ref{D2}) we finish the proof.
\end{proof}
Recall that
$$Y(z)=\{y\in Y\mid  \arg\max_{y\in Y}d(y, z)\}$$
Note that
$$P(z)=d({y}(z), z),$$
for any ${y}(z)\in Y(z)$.
\begin{lemma}[Proximal Descent]\label{proximal-descent}
For any $t\ge 0$, there holds
\begin{equation}\label{eq:prox-descent}
P(z^{t+1})-P(z^{t})\le \frac{p}{2}(z^{t+1}-z^t)^T(z^t+z^{t+1}-2x({y}(z^{t+1}, z^t)),
\end{equation}
where ${y}(z^{t+1})$ is arbitrary $y$ belongs to the set $Y(z^{t+1})$.
\end{lemma}
\begin{proof}
In the rest of the Appendix, ${y}(z^{t+1})$ denote any $y$ belongs to the set $Y(z^{t+1})$.
Using Kakutoni's Theorem, we have
$$\min_{x\in X}\max_{y\in Y}K(x, z;y)=\max_{y\in Y}\min_{x\in X}K(x, z;y),$$
which implies
$$\max_{y\in Y}d(y, z)=\min_{x\in X}h(x, z)=P(z).$$
Hence we have
\begin{eqnarray*}
&&P(z^{t+1})-P(z^t)\\
&\stackrel{ \mbox{\scriptsize(i)}}\le&P(z^{t+1})-d({y}(z^{t+1}), z^t)\\
&\stackrel{ \mbox{\scriptsize(ii)}}=&d({y}(z^{t+1}), z^{t+1})-d({y}(z^{t+1}), z^t)\\
&\stackrel{ \mbox{\scriptsize(iii)}}=&K(x({y}(z^{t+1}), z^t), z^{t+1}; {y}(z^{t+1}))-K(x({y}(z^{t+1}), z^t), z^t; {y}(z^{t+1}))\\
&\stackrel{ \mbox{\scriptsize(iv)}}=&\frac{p}{2}(z^{t+1}-z^t)^T(z^{t+1}+z^t-2x({y}(z^{t+1}), z^t)),
\end{eqnarray*}
where (i) and (iii) are because of \eqref{weak-duality}, (ii) is because of the definition of $y(z^{t+1})$ and (iv) is from direct calculation.
\end{proof}

\begin{lemma}\label{eqvl}
$x^*(z)=x(y, z)$ for any $y\in Y(z)$ and $x^*(z)$ is continuous in $z$.
Moreover, we have
$$z=x^*(z)$$
if and only if $z\in X^*$.
\end{lemma}
We define
$$y_+(z)=P_Y(y+\alpha \nabla_yK(x(y, z), z;y)).$$
Then we have
\begin{lemma}\label{yplus-appe}
We have
$$\|y^{t+1}-y^t_+(z^t)\|\le \kappa\|x^t-x^{t+1}\|,$$
where $\kappa=\alpha L\sigma_3$.
\end{lemma}
\begin{proof}
By the nonexpansiveness of the projection operator, we have
\begin{eqnarray*}
\|y^{t+1}-y^t_+(z^t)\|&=&\|P_Y(y^t+\alpha \nabla_yK(x(y^t, z^t), z^t;y^t))-P_Y(y^t+\alpha \nabla_yK(x^{t+1}, z^t;y^t))\|\\
&\le&\|(y^t+\alpha \nabla_yK(x(y^t, z^t), z^t;y^t))-(y^t+\alpha\nabla_yK(x^{t+1}, z^t; y^t))\|\\
&\le&\alpha L\|x^{t+1}-x(y^t, z^t)\|\\
&\le&\alpha L\sigma_3\|x^t-x^{t+1}\|,
\end{eqnarray*}
where the first inequality is due to the nonexpansiveness of the projection operator, the second is because of the Lipschitz-continuity of $\nabla_yK$ and the last inequality is because of \eqref{eb3}.
\end{proof}

Now we can  prove Proposition \ref{basic-estimate-appe}.

\begin{proof}
Using the three descent lemma, we have
\begin{eqnarray}\label{the_eqn1}
\nonumber&&\Phi^t-\Phi^{t+1}\\
\nonumber&\ge &\frac{1}{2c}\|x^t-x^{t+1}\|^2+ \langle\nabla_y K(x^{t+1}, z^t; y^t), y^t - y^{t+1}\rangle
-\frac{L}{2}\|y^t-y^{t+1}\|^2+\frac{p}{2\beta}\|z^t-z^{t+1}\|^2\\
\nonumber&&+2\langle \nabla_y K(x(y^t, z^t),z^t;y^t),y^{t+1}-y^t\rangle-\frac{2L_d}{2}\|y^t-y^{t+1}\|^2
+p(z^{t+1}-z^t)^T(z^{t+1}+z^t-2x(y^{t+1}, z^{t+1}))\\
\nonumber&&-p(z^{t+1}-z^t)^T(z^t+z^{t+1}-2x({y}(z^{t+1}), z^t))\\
\nonumber&=&\frac{1}{2c}\|x^t-x^{t+1}\|^2-\frac{L+2L_d}{2}\|y^t-y^{t+1}\|^2+\frac{p}{2\beta}\|z^t-z^{t+1}\|^2 +\langle \nabla_y K(x^{t+1},z^t;y^t),y^{t+1}-y^t\rangle \\
&&+2\langle \nabla_yK(x(y^t, z^t), z^t ;y^t)-\nabla_yK(x^{t+1}, z^t; y^t), y^{t+1}-y^t\rangle\\
&&+2p(z^{t+1}-z^t)^T(x({y}(z^{t+1}), z^{t})-x(y^{t+1}, z^{t+1})).
\end{eqnarray}
Using the property of the projection operator and the update of the dual variable $y^{t}$, we have
\begin{eqnarray*}
\langle \nabla_y K(x^{t+1},z^t;y^t),y^{t+1}-y^t\rangle \geq \frac{1}{\alpha}\|y^t-y^{t+1}\|^2.
\end{eqnarray*}
Substituting the above inequality into (\ref{the_eqn1}),  we get
\begin{eqnarray*}
&&\Phi^t-\Phi^{t+1}\\
&\ge&\frac{1}{2c}\|x^t-x^{t+1}\|^2+(\frac{1}{\alpha}-\frac{L+2L_d}{2})\|y^t-y^{t+1}\|^2+\frac{p}{2\beta}\|z^t-z^{t+1}\|^2\\
&&+2\langle \nabla_yK(x(y^t, z^t), z^t ;y^t)-\nabla_yK(x^{t+1}, z^t; y^t), y^{t+1}-y^t\rangle\\
&&+2p(z^{t+1}-z^t)^T(x({y}(z^{t+1}), z^{t})-x(y^{t+1}, z^{t+1}))\\
&\ge&\frac{1}{2c}\|x^t-x^{t+1}\|^2+(\frac{1}{\alpha}-\frac{L+10L}{2})\|y^t-y^{t+1}\|^2+\frac{p}{2\beta}\|z^t-z^{t+1}\|^2\\
&&+2\langle \nabla_yK(x(y^t, z^t), z^t ;y^t)-\nabla_yK(x^{t+1}, z^t; y^t), y^{t+1}-y^t\rangle\\
&&+2p(z^{t+1}-z^t)^T(x({y}(z^{t+1}), z^{t})-x(y^{t+1}, z^{t+1}))\\
&\ge&\frac{1}{2c}\|x^t-x^{t+1}\|^2+\frac{1}{2\alpha}\|y^t-y^{t+1}\|^2+\frac{p}{2\beta}\|z^t-z^{t+1}\|^2\\
&&+2\langle \nabla_yK(x(y^t, z^t), z^t ;y^t)-\nabla_yK(x^{t+1}, z^t; y^t), y^{t+1}-y^t\rangle\\
&&+2p(z^{t+1}-z^t)^T(x({y}(z^{t+1}), z^{t})-x(y^{t+1}, z^{t+1}))
\end{eqnarray*}
where the second inequality is because of \eqref{5L} and the last inequality is because $\alpha\leq \frac{1}{11L}$.

Notice that
\begin{eqnarray*}
&&2p(z^{t+1}-z^t)^T(x({y}(z^{t+1}), z^t)-x(y^{t+1}, z^{t+1}))\\
&=&2p(z^{t+1}-z^t)^T((x({y}(z^{t+1}), z^t)-x({y}(z^{t+1}), z^{t+1}))+(x({y}(z^{t+1})), z^{t+1})-x(y^{t+1}, z^{t+1})))\\
&=&2p(z^{t+1}-z^t)^T(x({y}(z^{t+1}), z^t)-x({y}(z^{t+1}), z^{t+1}))\\
&&+2p(z^{t+1}-z^t)^T(x({y}(z^{t+1})), z^{t+1})-x(y^{t+1}, z^{t+1}))\\
&\stackrel{ \mbox{\scriptsize(i)}}\ge&-2p\sigma_1\|z^{t+1}-z^t\|^2+2p(z^{t+1}-z^t)^T(x({y}(z^{t+1})), z^{t+1})-x(y^{t+1}, z^{t+1}))\\
&\stackrel{ \mbox{\scriptsize(ii)}}\ge&-2p\sigma_1\|z^{t+1}-z^t\|^2-\frac{p}{6\beta}\|z^{t+1}-z^t\|^2-6p\beta\|x({y}(z^{t+1}), z^{t+1})-x(y^{t+1}, z^{t+1})\|^2,
\end{eqnarray*}
where (i) is because of the Cauchy-Schwarz inequality and Lemma \ref{Lp_continue} and (ii) is due to the AM-GM inequality.
Also we have
\begin{eqnarray*}
&&2\langle \nabla_yK(x(y^t, z^t), z^t ;y^t)-\nabla_yK(x^{t+1}, z^t; y^t), y^{t+1}-y^t\rangle\\
&\ge&-2\|\nabla_yK(x(y^t, z^t), z^t ;y^t)-\nabla_yK(x^{t+1}, z^t; y^t)\|\cdot\| y^{t+1}-y^t\|\\
&\ge&-2L\|x^{t+1}-x(y^t, z^t)\|\cdot\|y^t-y^{t+1}\|\\
&\ge&-L\sigma_3^2\|y^t-y^{t+1}\|^2-L\sigma_3^{-2}\|x^{t+1}-x(y^t, z^t)\|^2\\
&\ge&-L\sigma_3^2\|y^t-y^{t+1}\|^2-L\|x^{t+1}-x^t\|^2,
\end{eqnarray*}
where the first inequality is because of the Cauchy-Schwarz in equality, the second inequality is because $\nabla_yK=\nabla_yf$ is $L$-Lipschitz-continuous, the third inequality is due to the AM-GM inequality and the last is because of \eqref{eb3}.

Hence we have
\begin{eqnarray}
\nonumber&&\Phi^t-\Phi^{t+1}\\
\nonumber&\ge&(\frac{1}{2c}-L)\|x^t-x^{t+1}\|^2+(\frac{1}{2\alpha}-L\sigma_3^2)\|y^t-y^{t+1}\|^2
+(\frac{p}{2\beta}-2p\sigma_1-\frac{p}{6\beta})\|z^t-z^{t+1}\|^2\\
\nonumber&&-6p\beta\|x({y}(z^{t+1}), z^{t+1})-x(y^{t+1}, z^{t+1})\|^2
\end{eqnarray}
By the conditions of $p, c$, we have
$$1/2c-L\ge 1/4c.$$
By the condition for $\alpha$, we have
$$\alpha<1/(4L\sigma_3^2),$$
which yields
$$1/(2\alpha)-L\sigma_3^2\ge 1/(4\alpha)$$
And by the conditions that $\beta<\frac{1}{36}$ and $p\ge 3L$ together with the definition of $\sigma_1$,
$$
\frac{p}{2\beta}-2p\sigma_1-\frac{p}{6\beta}\geq\frac{p}{4\beta}.
$$
Then we have
\begin{eqnarray}\label{original-estimate}
\nonumber&&\Phi^t-\Phi^{t+1}\\
\nonumber&\ge&\frac{1}{4c}\|x^t-x^{t+1}\|^2+\frac{1}{4\alpha}\|y^t-y^{t+1}\|^2
+\frac{p}{4\beta}\|z^t-z^{t+1}\|^2\\
\label{Phiphi}&&-6p\beta\|x({y}(z^{t+1}), z^{t+1})-x(y^{t+1}, z^{t+1})\|^2\\
\nonumber&=&\frac{1}{4c}\|x^t-x^{t+1}\|^2+\frac{1}{4\alpha}\|y^t-y^{t+1}\|^2
+\frac{p}{4\beta}\|z^t-z^{t+1}\|^2\\
&&-6p\beta\|x^*(z^{t+1})-x(y^{t+1}, z^{t+1})\|^2,
\end{eqnarray}
where  the last equality is because of Lemma \ref{eqvl}.
By Lemma \ref{yplus-appe} and the convexity of the norm square function, we have
\begin{eqnarray}
\|y^{t+1}-y^t\|^2&=&\|(y^{t+1}-y^t_+(z^t))+(y^t_+(z^t)-y^t)\|^2\\
&\ge&\|y^t-y^t_+(z^t)\|^2/2-\|y^{t+1}-y^t_+(z^t)\|^2\\
&\ge&\|y^t-y^t_+(z^t)\|^2/2-\kappa^2\|x^t-x^{t+1}\|^2.\label{spl1}
\end{eqnarray}
Similarly, by  Lemma \ref{yplus-appe},   \eqref{eb2} and the convexity of norm square function, we have
\begin{eqnarray}
&&\|x^*(z^{t+1})-x(y^{t+1}, z^{t+1})\|^2\\
&=&\|(x^*(z^{t+1})-x^*(z^{t}))+(x^*(z^t)-x(y^t_+(z^t), z^t))\\
&&+(x(y^t_+(z^t), z^t)-x(y^{t+1}, z^t))+(x(y^{t+1}, z^t)-x(y^{t+1}, z^{t+1}))\|^2\\
&\le&4\|x^*(z^{t+1})-x^*(z^{t})\|^2+4\|x^*(z^t)-x(y^t_+(z^t), z^t)\|^2\\
&&+4\|x(y^t_+(z^t), z^t)-x(y^{t+1}, z^t)\|^2+4\|x(y^{t+1}, z^t)-x(y^{t+1}, z^{t+1})\|^2\\
&\le&4\sigma_1^2\|z^t-z^{t+1}\|^2+4\|x^*(z^t)-x(y^t_+(z^t), z^t)\|^2\\
&&+4\sigma_2^2\kappa^2\|x^t-x^{t+1}\|^2+4\sigma_1^2\|z^t-z^{t+1}\|^2\\
&=&8\sigma_1^2\|z^t-z^{t+1}\|^2+4\|x^*(z^t)-x(y^t_+(z^t), z^t)\|^2\\
&&+4\sigma_2^2\kappa^2\|x^t-x^{t+1}\|^2.\label{spl2}
\end{eqnarray}
Substituting \eqref{spl1} and \eqref{spl2} to \eqref{original-estimate} yields
\begin{eqnarray*}
&&\phi^t-\phi^{t+1}\\
\nonumber&\ge&(\frac{1}{4c}-24p\beta\sigma_2^2\kappa^2-\kappa^2/(4\alpha))\|x^t-x^{t+1}\|^2\\
&&+\frac{1}{8\alpha}\|y^t-y^t_+(z^t)\|^2+(\frac{p}{4\beta}-48p\beta\sigma_1^2)\|z^t-z^{t+1}\|^2\\
&&-24p\beta\|x^*(z^{t})-x(y^t_+(z^t), z^{t})\|^2.
\end{eqnarray*}
Notice that
$$\alpha<1/(4L\sigma_3^2)<1/(4cL^2\sigma_3^2),$$
and hence $\kappa^2/(4\alpha)=\alpha^2L^2\sigma_3^2/(4\alpha)<1/(16c)$.
Also we have
$$\beta<1/(96p\alpha\sigma_2^2),$$
thus
$$24p\beta\sigma_2^2\kappa^2<\kappa^2/(4\alpha)\le 1/(16c).$$
Consequently, we have
$$(\frac{1}{4c}-24p\beta\sigma_2^2\kappa^2-\kappa^2/(4\alpha))>1/(8c).$$
By the definition of $\sigma_1$ and the conditions $p\ge 3L, \beta<\frac{1}{36}$, we have
$$(\frac{p}{4\beta}-48p\beta\sigma_1^2)\ge \frac{p}{8\beta}.$$
Combining the above, we have
\begin{eqnarray*}
&&\phi^t-\phi^{t+1}\\
\nonumber&\ge&\frac{1}{8c}\|x^t-x^{t+1}\|^2\\
&&+\frac{1}{8\alpha}\|y^t-y^t_+(z^t)\|^2+\frac{p}{8\beta}\|z^t-z^{t+1}\|^2\\
&&-24p\beta\|x^*(z^{t})-x(y^t_+(z^t), z^{t})\|^2,
\end{eqnarray*}
which finishes the proof.
\end{proof}

\subsection{General nonconvex-concave case}

We have the following error bound:
\begin{lemma}\label{weak error bound-appe}
We have
\begin{eqnarray*}
\alpha (p-L)\|x^*(z)-x(y_+(z), z)\|^2
&<& (1+\alpha L + \alpha L \sigma_2) \|y-y_+(z)\|\cdot \mathrm{dist}(y_+(z), Y(z)).\\
&\le&(1+\alpha L + \alpha L \sigma_2)\|y-y_+(z)\|\cdot D(Y),
\end{eqnarray*}
where $D(Y)$ is the diameter of $Y$.
\end{lemma}
The proof will be given in the \cref{Appendix: error bound}.
\begin{lemma}\label{trivial1}
If
\begin{equation}\label{key4-appe}
\max\{\|x^t-x^{t+1}\|, \|y^t-y^t_+(z^t)\|, \|z^t-x^{t+1}\|\}\le \epsilon,
\end{equation}
then $(x^{t+1}, y^{t+1})$ is a $\bar{\lambda}\epsilon-$ solution for some $\bar{\lambda}>0$.
\end{lemma}
\begin{proof}

By the update of $x^{t+1}$,  we have
$$x^{t+1}=\arg\min_{x\in X}\{\langle \nabla_xf(x^t, y^t)+p(x^t-z^t), x-x^t\rangle+\frac{1}{2c}\|x-x^t\|^2+\partial {\mathbf{1}_X(x)}\}.$$
Therefore, we have
\begin{equation}\label{trivial11}
0\in \nabla_xf(x^{t}, y^t)+p(x^{t+1}-z^t)+\frac{1}{2c}(x^{t+1}-x^t)+\partial {\mathbf{1}_X(x^{t+1})}.
\end{equation}
Similarly, we have
\begin{equation}\label{trivial12}
0\in -\nabla_yf(x^{t+1}, y^t)+\frac{1}{\alpha}(y^{t+1}-y^t)+\partial {\mathbf{1}_Y(y^{t+1})}.
\end{equation}

We let
$$u=(\nabla_xf(x^{t+1}, y^{t})-\nabla_xf(x^{t}, y^{t}))+(\nabla_xf(x^{t+1}, y^{t+1})-\nabla_xf(x^{t+1}, y^{t}))-p(x^{t+1}-z^t)-\frac{1}{2c}(x^{t+1}-x^t)$$
and
$$v=\nabla_yf(x^{t+1}, y^t)-\nabla_yf(x^{t+1}, y^{t+1})-\frac{1}{\alpha}(y^{t+1}-y^t).$$
By the Lipschitz-continuity of $\nabla_xf(x, y)$, Lemma \ref{yplus-appe} and \eqref{key4-appe}, we have
\begin{eqnarray*}
\|u\|&\le&L\|x^t-x^{t+1}\|+L\|y^t-y^{t+1}\|+p\epsilon+\frac{1}{2c}\epsilon\\
&\le&(L+p+1/2c)\epsilon+L\|y^t-y^t_+(z^t)\|+L\|y^t_+(z^t)-y^{t+1}\|\\
&\le&(L+p+1/2c)\epsilon+L\epsilon+L\kappa\epsilon = ((2 + \kappa)L + p + 1/2c) \epsilon,
\end{eqnarray*}
where the first and the second inequalities are both due to \eqref{key4-appe}, the triangular inequality and the Lipschitz-continuity of $\nabla_xf(\cdot)$ and the last inequality is because of Lemma \ref{yplus-appe}.
Similarly, we can prove that
$$\|v\|\le (L + \frac{1}{\alpha})(1 + \kappa)\epsilon.$$
Hence, we finish the proof with $\bar{\lambda}=(2 +\kappa)(L+ 1/\alpha) + p+ 1/2c$.

\end{proof}
We say that $\phi^t$ decreases sufficiently if
\begin{equation}\label{suff-decrease-appe}
\phi^t-\phi^{t+1}\ge \frac{1}{16c}\|x^t-x^{t+1}\|^2+\frac{1}{16\alpha}\|y^t-y^t_+(z^t)\|^2+\frac{p\beta}{16}\|z^t-x^{t
+1}\|^2.
\end{equation}
\begin{lemma}\label{trivial-appe}
Let $T>0$. Then if for any $t\in\{0, 1, \cdots, T-1\}$, \eqref{suff-decrease-appe} holds, there must exist a $t\in\{1, 2,\cdots, T\}$ such that $(x^t, y^t)$ is an $C/\sqrt{T\beta}$-solution.
Moreover, if for any $t\ge 0$,  \eqref{suff-decrease-appe} holds,
Then any limit point of $(x^t, y^t)$ is a solution of \eqref{minmax2}, and the iteration complexity of attaining an $
\epsilon-$solution is $\mathcal{O}(1/\epsilon^2)$.
\end{lemma}

\begin{proof}

We have
\begin{eqnarray}\label{suff}
\phi^0-\underline{f}&\ge&\sum_{t=0}^{T-1}(\phi^t-\phi^{t+1})\\
&\ge&\min\{1/(16c), 1/(16\alpha), p/16\}\sum_{t=0}^{T-1}\max\{\|x^t-x^{t+1}\|^2, \\
&&\|y^t-y^t_+(z^t)\|^2, 
\beta \|x^{t+1}-z^t\|^2\},
\end{eqnarray}
where the last inequality is due to \eqref{suff-decrease-appe}.
Therefore, there exists a $t\in\{0,1, \cdots,T-1\}$ such that
$$\min\{1/(16c), 1/(16\alpha), p/16\}\max\{\|x^t-x^{t+1}\|^2, \|y^t-y^t_+(z^t)\|^2, \beta\|x^{t+1}-z^t\|^2\}\le (\phi^0-\underline{f})/T.$$
Since $\beta<1$, we further get
$$\min\{1/(16c), 1/(16\alpha), p/16\}\max\{\|x^t-x^{t+1}\|^2, \|y^t-y^t_+(z^t)\|^2, \|x^{t+1}-z^t\|^2\}\le (\phi^0-\underline{f})/(T\beta).$$
Hence, by Lemma \ref{trivial1}, $(x^{t+1}, y^{t+1})$ is a $\bar{\lambda} \sqrt{(\phi^0-\underline{f})/T\beta}$-solution, where $$\bar{\lambda} = 16 \left(\left(2 + \kappa\right)\left(L + \frac{1}{\alpha}\right) + p + \frac{1}{2c}\right) \max\{c, \alpha, \frac{1}{p}\}.$$
According to above analysis, If  \eqref{suff-decrease-appe} holds for any $t$,  we can attain an $\epsilon$-solution within
$$ \frac{ 16(\phi^0-\underline{f})\epsilon^2}{\beta} \max\{c, \alpha, 1/p\}$$
iterations.
Moreover, if \eqref{suff-decrease-appe} holds for any $t$, by \eqref{suff}, we have
\begin{equation}
\max\{\|x^t-x^{t+1}\|, \|y^t-y^t_+(z^t)\|, \|z^t-x^{t+1}\|\}\rightarrow 0.
\end{equation}
Consequently, for any limit point $(\bar{x}, \bar{y})$ of $(x^t, y^t)$, there exists a $\bar{z}$ such that
$$\max\{\|\bar{x}-\bar{x}_+(\bar{y}, \bar{z})\|, \|\bar{y}-\bar{y}_+(\bar{z})\|, \|\bar{x}_+(\bar{y}, \bar{z})-\bar{z}\|\}=0,$$
which yields $(\bar{x}, \bar{y})$ is a stationary solution.
Here  $$x_+(y, z)=P_X(x-\nabla_xK(x, z; y)).$$
\end{proof}
Now we are ready to prove Theorem \ref{general}.

\begin{proof}[Proof of  Theorem \ref{general}]
There are two cases \eqref{case1-appe} and \eqref{case2} as discussed in the proof for the general nonconvex-concave problems in last subsection.
 \begin{enumerate}
 \item  For some $t\in \{0, 1, \cdots, T-1\}$, we have
 \begin{equation}\label{case1-appe}
 \frac{1}{2}\max\{\frac{1}{8c}\|x^t-x^{t+1}\|^2, \frac{1}{8\alpha}\|y^t-y^t_+(z^t)\|^2, \frac{p}{8\beta}\|z^t-z^{t+1}\|^2\}
 \le 24p\beta\|x^*(z^t)-x(y^t_+(z^t), z^t)\|^2.
 \end{equation}
 \item For any $t\in \{0, 1, \cdots, T-1\}$,
 \begin{equation}\label{case2}
 \frac{1}{2}\max\{\frac{1}{8c}\|x^t-x^{t+1}\|^2, \frac{1}{8\alpha}\|y^t-y^t_+(z^t)\|^2, \frac{p}{8\beta}\|z^t-z^{t+1}\|^2\}
 \ge 24p\beta\|x^*(z^t)-x(y^t_+(z^t), z^t)\|^2
 \end{equation}
 \end{enumerate}
 
In the first case \eqref{case1-appe}, we have
\begin{eqnarray*}
\|y^t-y^t_+(z^t)\|^2&\le& 384p\beta\alpha\|x(y^t_+(z^t), z^t)-x^*(z^t)\|^2\\
&\le& 384p\beta\alpha\frac{1+\alpha L + \alpha L \sigma_2}{\alpha (p-L)}\|y^t-y^t_+(z^t)\|D(Y).\label{key3.3}
\end{eqnarray*}
Hence, letting
$\lambda_1=384p\frac{(1+\alpha L + \alpha L \sigma_2)}{p-L}\cdot D(Y)$, 
we have
\begin{equation}\label{key1}
\|y^t-y^t_+(z^t)\|\le \lambda_1\beta.
\end{equation}

Moreover,
\begin{eqnarray}\label{key2}
\|x^{t+1}-z^t\|^2&=&\|(z^{t+1}-z^t)/\beta\|^2\\
&\stackrel{ \mbox{\scriptsize(i)}} \le&384p\|x(y^t_+(z^t), z^t)-x^*(z^t)\|^2\\
&\stackrel{ \mbox{\scriptsize(ii)}} \le&384p\frac{1+\alpha L + \alpha L \sigma_2}{\alpha (p-L)}D(Y)\|y^t-y^t_+(z^t)\|\\
&\stackrel{ \mbox{\scriptsize(iii)}} \le&384p\frac{1+\alpha L + \alpha L \sigma_2}{\alpha (p-L)}D(Y)\lambda_1\beta,
\end{eqnarray}
where Inequality (i) is due to Inequality \eqref{case1-appe} and (ii)  is because of Lemma \ref{weak error bound-appe} and (iii) is
due to \eqref{key1}.
We also have
\begin{eqnarray}\label{key3}
\|x^t-x^{t+1}\|^2
&\stackrel{ \mbox{\scriptsize(i)}} \le&384cp\beta\|x^*(z^t)-x(y^t_+(z^t), z^t)\|^2\\
&\stackrel{ \mbox{\scriptsize(ii)}} \le&384pc\beta\frac{1+\alpha L + \alpha L \sigma_2}{\alpha (p-L)}D(Y)\|y^t-y^t_+(z^t)\|\\
&\stackrel{ \mbox{\scriptsize(iii)}} \le&384pc\frac{1+\alpha L + \alpha L \sigma_2}{\alpha (p-L)}\lambda_1D(Y)\beta^2,
\end{eqnarray}
where (i) is  due to \eqref{case1-appe}, (ii) is due to Lemma \ref{weak error bound-appe} and (iii) is because of \eqref{key1}.
Combining the above, in the first case, we have
\begin{eqnarray}\label{beta-bound}
&&\max\{\|x^t-x^{t+1}\|^2, \|y^t-y^t_+(z^t)\|^2, \|z^t-x^{t+1}\|^2\}\\
&\le& \max\{\lambda_2\beta^2, \lambda_1^2\beta^2, \lambda_3\beta\},
\end{eqnarray}
where $\lambda_2=384p\frac{1+\alpha L + \alpha L \sigma_2}{\alpha (p-L)}D(Y)\lambda_1$ and $\lambda_3=192pc\frac{1+\alpha L + \alpha L \sigma_2}{\alpha (p-L)}\lambda_1D(Y)$
According to Lemma \ref{trivial1}, there exists a $\lambda>0$ such that $(x^{t+1}, y^{t+1})$ is a $\lambda\max\{\beta, \sqrt{\beta}\}$-solution.

In the second case, we have
$$\phi^t-\phi^{t+1}\ge \frac{1}{16c}\|x^t-x^{t+1}\|^2+\frac{1}{16\alpha}\|y^t-y^t_+(z^t)\|^2+\frac{1}{16\beta}\|z^t-z^{t
+1}\|^2$$
for any $t\in\{0, 1, \cdots, T-1\}$.
By Lemma \ref{trivial-appe}, there exists a $t\in\{0, 1,\cdots, T-1\}$, 
such that $(x^{t+1}, y^{t+1})$ is a
$\bar{\lambda} \sqrt{(\phi^0-\underline{f})/T\beta}$-solution, where $\bar{\lambda} = 16 \left(\left(2 + \kappa\right)\left(L + \frac{1}{\alpha}\right) + p + \frac{1}{2c}\right) \max\{c, \alpha, \frac{1}{p}\}.$
Finally taking $\beta=1/\sqrt{T}$ and combining the two cases  with Lemma \ref{trivial-appe} yield the desired results.

\end{proof}

\subsection{The max problem is over a discrete set}
In this subsection, we prove Theorem \ref{discrete}. We will prove that under the strict complementarity assumption, the
potential function $\phi^t$ decreases sufficiently after any iteration. Then by the following simple lemma, we can prove
Theorem \ref{discrete}.

By the bounded level set assumption (Assumption \ref{boundedness}) and the fact that $\psi(z)\le P(z)$, for any $(x^0, y^0, z^0)\in\mathbb{R}^{n+m+n}$, there exists a constant $R(x^0, y^0, z^0)
>0$ such that
$$\{z\mid P(z)\le \phi(x^0, y^0, z^0)\}\subseteq  \mathcal{B}(R(x^0, y^0, z^0)).$$
Then we have the following ``dual error bound''. Note that this error bound is homogeneous compared to Lemma \ref{weak error bound-appe}.
\begin{lemma}\label{dual error bound-appe}
Let $$x_+(y, z)=P_X(x-\nabla_xK(x, z; y)).$$
If the strict complementarity assumption and the bounded level set assumption hold for \eqref{minmax2} , there exists $\delta>0$, such that if
$$\|z\|\le R(x^0, y^0, z^0),$$
and
$$\max\{\|x-x_+(y, z)\|, \|y-y_+(z)\|, \|x_+(y, z)-z\|\}<\delta$$
we have
$$\|x(y_+(z), z)-x^*(z)\|< \sigma_5\|y-y_+(z)\|$$
for  some constant $\sigma_5>0$.

\end{lemma}
Equipped with the dual error bound, we can prove that the potential function decreases after any iteration in the
following proposition:
\begin{proposition}
Suppose the conditions in Theorem \ref{discrete} holds, we have
\begin{equation}
\phi^t-\phi^{t+1}\ge \frac{1}{16c}\|x^t-x^{t+1}\|^2+\frac{1}{16\alpha}\|y^t-y^t_+(z^t)\|^2+\frac{p}{16\beta}\|z^t-z^{t
+1}\|^2.
\end{equation}
\end{proposition}
\begin{proof}
We set $\beta<\min\{\delta/\sqrt{\lambda_2}, \delta/\lambda_1, \delta^2/\lambda_3, 1/(384p\alpha\sigma_5^2)\}$.
First, we prove that
\begin{equation}
\phi^t-\phi^{t+1}\ge \frac{1}{16c}\|x^t-x^{t+1}\|^2+\frac{1}{16\alpha}\|y^t-y^t_+(z^t)\|^2+\frac{p}{16\beta}\|z^t-z^{t
+1}\|^2.
\end{equation}
and
$$\|z^t\|< R(x^0, y^0, z^0)$$
for any $t\ge 0$.
We prove it by induction. We will prove that
\begin{enumerate}
\item If $\|z^t\|\le R(x^0,y^0, z^0)$, then
\begin{equation}
\phi^t-\phi^{t+1}\ge \frac{1}{16c}\|x^t-x^{t+1}\|^2+\frac{1}{16\alpha}\|y^t-y^t_+(z^t)\|^2+\frac{p}{16\beta}\|z^t-z^{t
+1}\|^2.
\end{equation}
\item If $\phi^{t+1}\le \phi^t$, we have $\|z^{t+1}\|\le R(x^0, y^0, z^0)$.
\end{enumerate}

For $t=0$, it is trivial that $\|z^t\|\le R(x^0, y^0, z^0)$.
For the first step, assume that we have $\|z^t\|\le R(x^0, y^0, z^0)$.
There are two cases:
\begin{enumerate}
\item  For some $t$, we have
\begin{equation}\label{case1-appe2}
\frac{1}{2}\max\{\frac{1}{8c}\|x^t-x^{t+1}\|^2, \frac{1}{8\alpha}\|y-y^t_+(z^t)\|^2, \frac{p}{8\beta}\|z^t-z^{t+1}\|^2\}
\le 24p\beta\|x^*(z^t)-x(y^t_+(z^t), z^t)\|^2.
\end{equation}
\item For any $t$,
\begin{equation}\label{case2-appe2}
\frac{1}{2}\max\{\frac{1}{8c}\|x^t-x^{t+1}\|^2, \frac{1}{8\alpha}\|y-y^t_+(z^t)\|^2, \frac{p}{8\beta}\|z^t-z^{t+1}\|^2\}
\ge 24p\beta\|x^*(z^t)-x(y^t_+(z^t), z^t)\|^2
\end{equation}
\end{enumerate}
For the first case, as in the last subsection, we have
\begin{eqnarray*}
&&\max\{\|x^t-x^{t+1}\|^2, \|y^t-y^t_+(z^t)\|^2, \|x^{t+1}-z^t\|^2\}\\
&\le&\max\{\lambda_2\beta^2, \lambda_1^2\beta^2, \lambda_3\beta\}\\
&\le&\delta^2.
\end{eqnarray*}
Hence, we can make use of Lemma \ref{dual error bound-appe}.  In fact, we have
\begin{eqnarray*}
24p\beta\|x(y^t_+(z^t), z^t)-x^*(z^t)\|^2&\le& 24p\beta\sigma_5^2\|y^t-y^t_+(z^t)\|\\
&\le& \frac{1}{16\alpha}\|y^t-y^t_+(z^t)\|^2,
\end{eqnarray*}
which yields \eqref{suff-decrease-appe} together with \eqref{basic estimate-appe}.
For the second case, \eqref{suff-decrease-appe} holds as in the last subsection.
Hence, if $\|z^t\|\le R(x^0, y^0, z^0)$, we have \eqref{suff-decrease-appe}.
For the second step, if \eqref{suff-decrease-appe} holds for $0, 1, \cdots, (t-1)$, we have
\begin{eqnarray*}
P(z^{t+1})&\le& \phi^{t+1}\\
&\le&\phi^0.
\end{eqnarray*}
Hence, $z^{t+1}\in \mathcal{B}(R(x^0, y^0, z^0))$.
Combining these, for any $t\ge 0$, $\|z^t\|\le R(x^0, y^0, z^0)$ and \eqref{suff-decrease-appe} holds.
Then the theorem comes from Lemma \ref{trivial-appe}.
\end{proof}

\section{The multi-block cases}\label{Appendix: multi block}
The proofs for the multi-block case is similar to the one-block case. In this section, we briefly introduce the proof of them.
Note that the only differences for proving the theorem s are Lemma \ref{primal}, \eqref{eb3} and Proposition \ref{basic-estimate}.
Instead, we have the following:
\begin{lemma}[Primal Descent] \label{primal-multiblock}
For any $t$, we have
\begin{eqnarray}\label{K02}
\nonumber K(x^t, z^t; y^t)-K(x^{t+1}, z^{t+1}; y^{t+1})&\ge& \frac{1}{2c}\|x^t-x^{t+1}\|^2+ \langle\nabla_y K(x^{t+1}, z^t; y^t), y^t - y^{t+1}\rangle\\
&&-\frac{L}{2}\|y^t-y^{t+1}\|^2+\frac{p}{2\beta}\|z^t-z^{t+1}\|^2.
\end{eqnarray}
\end{lemma}
The proof of it is the same as Lemma 5.3 in \cite{zhang2020proximal}.
The error bound \eqref{eb3} becomes:
\begin{lemma}\label{eb32}
We have
$$\|x^{t+1}-x(y^t, z^t)\|\le \sigma_3'\|x^t-x^{t+1}\|,$$
where $\sigma_3'=(c(p-L)+1+c(L+p)N^{3/2})/c(p-L)$.
\end{lemma}
The proof of Lemma \ref{eb32} is similar  to Lemma 5.2 in \cite{zhang2020proximal} hence omitted here.
Because of the above two differences, we have a replacement of Proposition \ref{basic-estimate-appe}:
\begin{proposition}\label{basic2}
We let
\begin{align}\label{eqn: parameters for Alg 2 general case-appe2}
    p>3L, c<\frac{1}{p+L}, \alpha<\min\{\frac{1}{11L}, \frac{c^2(p-L)^2}{4L(1+c(p+L)N^{3/2}+c(p-L))^2}\}, \min\{\frac{1}{36}, \beta<\frac{(p-L)^2}{384p(p+L)^2}\}.
\end{align}
Then we have
\begin{eqnarray}\label{basic estimate-appe2}
&&\phi^t-\phi^{t+1}\\
\nonumber&\ge&\frac{1}{4c}\|x^t-x^{t+1}\|^2\\
&&+\frac{1}{4\alpha}\|y^t-y^t_+(z^t)\|^2+\frac{p}{8\beta}\|z^t-z^{t+1}\|^2\\
&&-24p\beta\|x^*(z^{t})-x(y^t_+(z^t), z^{t})\|^2
\end{eqnarray}
\end{proposition}
The proof of Proposition \ref{basic2}  is similar to Proposition \ref{basic-estimate-appe} hence omitted.
\section{Proof of the error bound lemmas}\label{Appendix: error bound}
\subsection{Proof of lemma \ref{weak error bound-appe}}
Let
$$x_+(y, z)=P_X(x-c\nabla_xK(x, z; y))$$
and
$$y_+(z)=P_Y(y+\alpha\nabla_yK(x(y, z), z;y)).$$
Then Lemma \ref{weak error bound-appe} can be written as

\begin{lemma}\label{weak error bound:rewrite}
We have
\begin{eqnarray*}
\alpha (p-L)\|x^*(z)-x(y_+(z), z)\|^2
&<& (1+\alpha L + \alpha L \sigma_2) \|y-y_+(z)\|\cdot \mathrm{dist}(y_+(z), Y(z)).\\
&\le&(1+\alpha L + \alpha L \sigma_2)\|y-y_+(z)\|\cdot D(Y),
\end{eqnarray*}
where $D(Y)$ is the diameter of $Y$.
\end{lemma}

\begin{proof}
By the strong convexity of $K(\cdot, z; y)$, we have
\begin{eqnarray}\label{gdforx}
K(x^*(z), z,; y_+(z))-K(x(y_+(z), z) z; y_+(z))&\ge&\frac{p-L}{2}\|x(y_+(z), z)-x^*(z)\|^2\\
K(x(y_+(z), z), z; y(z))-K(x^*(z), z; y(z))&\ge&\frac{p-L}{2}\|x(y_+(z), z)-x^*(z)\|^2,
\end{eqnarray}
where $y(z)$ is an arbitrary vector in $Y(z)$. Notice that $y_+(z)$ is the maximizer of the following problem 
(inspired by the proof in \cite{zhang2022global})
:
$$\max_{\bar{y}\in Y}\{\alpha K(x(y_+(z), z), z; \bar{y})-\delta^T(y, y_+(z); z)\bar{y}\},$$
where
$$\delta(y, y_+(z); z)=(y_+(z)+\alpha \nabla_{\bar{y}}K(x(y_+(z), z), z; y_+(z)))-(y+\alpha \nabla_{\bar{y}} K(x(y,z), z),  z; y))$$
satisfies
$$\|\delta(y, y_+(z); z)\|<(1+\alpha L + \alpha L \sigma_2)\|y-y_+(z)\|,$$
by the Lipschitz-continuity of $\nabla_yK=\nabla_yf$.
Hence, we have
\begin{eqnarray*}
&&\alpha K(x(y_+(z), z), z; y(z))-\delta^T(y, y_+(z); z)y(z)\\
&\le &
\alpha K(x(y_+(z), z), z; y_+(z))-\delta^T(y, y_+(z); z)y_+(z).
\end{eqnarray*}
Then, we  have the following estimates:

\begin{eqnarray}\label{gdfory}
&& \alpha (K(x(y_+(z), z), z; y(z))-K(x(y_+(z), z), z; y_+(z)))\\
&\le&(y(z)-y_+(z))^T\delta(y, y_+(z); z)\\
&\le&\|y_+(z)-y(z)\|\cdot (1+\alpha L + \alpha L \sigma_2)\|y-y_+(z)\|.
\end{eqnarray}
Also because $y(z)$ maximizes
$$\max_{\bar{y}\in Y}K(x^*(z), \bar{y}; z),$$
we have
\begin{equation}\label{gdfory2}
K(x^*(z), z; y(z))\ge K(x^*(z), z; y_+(z)).
\end{equation}
Since $y(z)$ is an arbitrary vector in $Y(z)$, combining \eqref{gdforx}, \eqref{gdfory}, \eqref{gdfory2}, we have
$$\alpha (p-L)\|x^*(z)-x(y_+(z), z)\|^2< (1+\alpha L + \alpha L \sigma_2)  \|y-y_+(z)\|\cdot \mathrm{dist}(y_+(z), Y(z)),$$
which is the desired result.
\end{proof}

\subsection{Proof of  Lemma \ref{dual error bound-appe}}
For a pair of min-max solution of \eqref{minmax2}, the KKT conditions in the following  hold:
\begin{eqnarray}\label{KKTfororiginal-appe}
&&J^TF(x^*)y=0,\\
&&\sum_{i=1}^my_i=1,\\
&&y_i\ge 0,\forall i\in [m]\\
&&\mu-\nu_i=f_i(x), \forall i\in[m],\\
&&\nu_i\ge 0, \nu_iy_i=0, \forall i\in[m],
\end{eqnarray}
where $\mu$ is the multiplier of the equality constraint $\sum_{i=1}^my_i=1$ and $\nu_i$ is the multiplier for the inequality constraint $y_i\ge 0$.
\begin{definition}
For $y\in Y$, we define the active set
$$\mathcal{A}[y]=\{i\in[m]\mid y_i=0\}.$$
We also define the inactive set of $y$ as follows:
$$\mathcal{I}[y]=\{i\in[m]\mid y_i>0\}.$$
\end{definition}

\begin{definition}
For an $x\in \mathbb{R}^n$, we define the top coordinate set $\mathcal{T}(x
)$ as the collection of all indexes of the top coordinates of $F(x)$, i.e., $f_i(x)>f_j(x)$ if $i\in \mathcal{T}(x), j\notin\mathcal{T}(x)$ and $f_i(x)=f_j(x)$ if $i, j\in\mathcal{T}(x)$.
\end{definition}
According to the KKT conditions, it is easy to see that for $(x, y)\in W^*$,
$$\mathcal{I}[y]\subseteq \mathcal{T}(x).$$

Recall that we have the following strict complementarity condition:
\begin{assumption}\label{sc-appe}
For  any $(x, y)$ satisfying  \eqref{KKTfororiginal-appe}, we have
$$\nu_i>0, \forall i\in\mathcal{A}[y].$$
\end{assumption}
It is easy to see that if the strict complementarity assumption holds,
$$\mathcal{I}[y]= \mathcal{T}(x)$$
for $(x, y)\in W^*$.
Then we can prove the following ``dual error bound''.
\begin{lemma}
If the strict complementarity assumption holds for \eqref{minmax2} , there exists $\delta>0$, such that if
$$\|z\|\le R(x^0, y^0, z^0),$$
and
$$\max\{\|x-x_+(y, z)\|, \|y-y_+(z)\|, \|x_+(y, z)-z\|\}<\delta$$
we have
$$\|x(y_+(z), z)-x^*(z)\|< \sigma_5\|y-y_+(z)\|$$
for  some constant $\sigma_5>0$.
\end{lemma}
To prove this, we
 need the following lemmas.
First, we prove that if the residuals go  to zero, the iteration  points converge to a solution.
\begin{lemma}\label{limit}
If $\{z^k\}$ is a sequence with $\|z^k\|\le R(x^0, y^0, z^0)$ and
$$\max\{\|x^k-x^k_+(y^k, z^k)\|, \|y^k-y^k_+(z^k)\|, \|x^k_+(y^k, z^k)-z^k\|\}\rightarrow 0,$$
there exists a sub-sequence of $\{z^k\}$ converging to some $\bar{z}\in X^*$.
\end{lemma}
\begin{proof}
It is just a direct corollary of Lemma \ref{trivial1}.
\end{proof}
\begin{lemma}\label{uniqueness}
Let
$$M(x)=\begin{Bmatrix}
J_{\mathcal{T}(x)}F(x)&\mathbf{1}
\end{Bmatrix}.$$
Then if $(x, y)\in W^*$, the matrix $M(x)$ is of full row rank.
\end{lemma}
\begin{proof}
We prove it by contradiction. If for some $(x^*, y^*, \mu^*, \nu^*)$ satisfying \eqref{KKTfororiginal-appe}, $M(x^*)$ is not of full row rank. Without loss of generality, we assume that $\mathcal{T}(x^*)=\{1, 2, \cdots, |\mathcal{T}(x^*)|\}$
Then there exists a nonzero vector $v\in \mathbb{R}^{|\mathcal{T}(x^*)|}$ such that
$$M^T(x^*)v=0.$$
Let $d=\min_{i\in\mathcal{I}[y^*]}\{y_i/|v_i|\}$. Then we define a vector $y'\in\mathbb{R}^m$ as:
\begin{eqnarray*}
y_i'=y^*-dv_i, &when&\ i\in \mathcal{I}[y^*];\\
y_i'=0,&when& \ otherwise.
\end{eqnarray*}
Notice that $y^*_i=0$ for any $i\notin \mathcal{T}(x^*)$. Then $y'$ satisfies
\begin{eqnarray*}
&&J^TF(x^*)y'=0,\\
&&\sum_{i=1}^my_i'=1,\\
&&y_i'\ge 0, i\in \mathcal{T}(x^*),\\
&&y_i'=0, i\notin\mathcal{T}(x^*).
\end{eqnarray*}

Therefore, $(x^*, y', \mu, \nu)$ still satisfies \eqref{KKTfororiginal-appe} .
Moreover, let $i_0\in\mathcal{I}[y^*]$ satisfying $d=y_{i_0}^*/v_{i_0}$.
Then $y_{i_0}'=\nu_{i_0}=0$. This is a contradiction to the strict complementarity assumption.
\end{proof}
We then have the following corollary from the above lemma and \eqref{KKTfororiginal-appe}:
\begin{corollary}\label{uniqueness1}
For any $x^*\in X^*$, there exists only one $y\in Y$ such that $(x^*, y)\in W^*$ and there exists only one $(\mu, \nu)$ such that $(x^*, y, \mu, \nu)$ satisfies \eqref{KKTfororiginal-appe}.
\end{corollary}
\begin{proof}
First, this $(y, \mu, \nu)$ must exist due to the existence of a solution.
Next, the solution $y$ must satisfy
$$M^T(x^*)y=(0, 0, \cdots, 0, 1)^T.$$
By  Lemma \ref{uniqueness}, $M^T(x^*)$ is of full column rank hence the solution of $y$ is unique.
Furthermore, since $\sum_{i=1}^my_i=1$, there is at least one $i$ such that $y_i>0, \nu_i=0$. Without loss of generality, we assume that $y_1>0, \nu_1=0$. Then  $\mu=f_1(x^*)$ by \eqref{KKTfororiginal}.
 Further by \eqref{KKTfororiginal}, $\nu_i=f_i(x^*)-\mu, i=2, 3, \cdots, m$.
Hence, $\mu, \nu_i$ are uniquely defined.
\end{proof}
\begin{lemma}\label{nonsingular}
If the strict complementarity assumption holds for \eqref{minmax2} , there exists $\delta>0, \gamma>0$, such that if
$$\|z\|\le R(x^0, y^0, z^0),$$
and
$$\max\{\|x-x_+(y, z)\|, \|y-y_+(z)\|, \|x_+(y, z)-z\|\}<\delta$$
$\gamma(M(x^*(z)))\ge \gamma$ and $\gamma(M(x(y, z)))\ge \gamma$.
\end{lemma}
\begin{proof}
We prove it by contradiction. Suppose it is not true, there exists $\{z^k\}\subseteq \mathcal{B}(R(x^0, y^0, z^0))$ such that $\gamma(M(x^*(z^k)))\rightarrow 0$ and
$$\max\{\|x^k-x^k_+(y^k, z^k)\|, \|y^k-y^k_+(z^k)\|, \|x^k_+(y^k, z^k)-z^k\|\}\rightarrow 0.$$

Since $\mathcal{T}(x)$ has only finite choice, without loss of generality, we assume that $\mathcal{T}(x^*(z^k))=\mathcal{T}$ for any $k$(passing to a sub-sequence if necessary).
By Lemma \ref{limit}, there exists a $\bar{z}\in X^*$ such that $z^k\rightarrow \bar{z}$.
We let
$$\tilde{M}(\bar{z})=\lim_{k\rightarrow \infty}M(x^*(z^k))=
\begin{Bmatrix}
J_{\mathcal{T}}(x^*(\bar{z}))&\mathbf{1}
\end{Bmatrix}.$$

By the continuity of $x^*(\cdot)$(\eqref{eb11} of Lemma \ref{Lp_continue}) and the continuity of the function of taking the least singular value, we know that
$$\gamma(\tilde{M}(\bar{z}))=0,$$

where we also use the fact that $x^*(\bar{z})=\bar{z}$ by Lemma \ref{eqvl}.
Moreover, according to the definition of $\mathcal{T}[\bar{z}]$, we have $f_i(x^*(z^k))>f_j(x^*(z^k))$ for any $k$ with $i\in \mathcal{T}, j\notin \mathcal{T}$.
Therefore, we have
$f_i(\bar{z})\ge f_j(\bar{z})$ for $i\in \mathcal{T}, j\notin \mathcal{T}$.
Consequently, we have
$$\mathcal{T}\subseteq \mathcal{T}[\bar{z}].$$

Therefore $\tilde{M}(x^*(\bar{z}))$ is a row sub-matrix of  $M(\bar{z})$. Consequently, $M(\bar{z})$ is not of full row rank. This is a contradiction!
For $x(y, z)$, it is similar to prove the desired result. Hence the details are omitted.
\end{proof}
The following lemma shows that if the residuals are small, the active set of $y_+(z)$ and $y(z)\in Y(z)$ are the same.
\begin{lemma}\label{same active set}
If the strict complementarity assumption holds for \eqref{minmax2} , there exists $\delta>0$, such that if
$$\|z\|\le R(x^0, y^0, z^0),$$
and
$$\max\{\|x-x_+(y, z)\|, \|y-y_+(z)\|, \|x_+(y, z)-z\|\}<\delta,$$
we have
$$\mathcal{A}[y_+(z)]=\mathcal{A}[y(z)], for\ some\ y(z)\in Y(z).$$
\end{lemma}
\begin{proof}
We prove it by contradiction. Suppose that there exists a sequence $\{(x^k, y^k, z^k)\}$ such that
$$\max\{\|x^k-x^k_+(y^k, z^k)\|, \|y^k-y^k_+(z^k)\|, \|x^k_+(y^k, z^k)-z^k\|\}\rightarrow 0$$
and
$$\mathcal{A}[y^k_+(z^k)]\ne \mathcal{A}[y(z^k)].$$
Since $\{y^k_+(z^k)\}, \{z^k\}$ are bounded, we assume that $y^k_+(z^k)\rightarrow \bar{y}, z^k\rightarrow \bar{z}$.
We write down the KKT condition for $(x(y^k_+(z^k), z^k), y^k_+(z^k))$ as follows:
\begin{eqnarray}\label{KKT1}
&&J^TF(x(y^k_+(z^k), z^k))y^k_+(z^k)+p(x(y^k_+(z^k), z^k)-z^k)=0,\\
&&\sum_{i=1}^m(y_i^k)_+(z^k)=1,\\
&&(y_i^k)_+(z^k)\ge 0,\forall i\in [m]\\
&&\frac{1}{\alpha}(y_i^k)_+(z^k)-\frac{1}{\alpha}y_i^k+ f_i(x(y^k, z^k))+\mu^k-\nu^k_i=f_i(x(y^k_+(z^k), z^k)), \forall i\in[m],\\
&&\nu^k_i\ge 0, \nu^k_i(y_i^k)_+(z^k)=0, \forall i\in[m],
\end{eqnarray}
It is not hard to check that $\mu, \nu$ are bounded.  Hence, we assume that $\mu^k\rightarrow \bar{\mu}$ and $\nu^k\rightarrow \bar{\nu}$.
We take limit to  \eqref{KKT1} and make use of  the fact that
$$\|y^k-y^k_+(z^k)\|\rightarrow 0$$
together with Lemma \ref{Lp_continue}. We then attain that $(x(\bar{y}, \bar{z}), \bar{y})$  is a min-max solution of \eqref{minmax2}, i.e., $(x(\bar{y}, bar{z}), \bar{y}, \bar{\mu}, \bar{\nu})$ satisfies \eqref{KKTfororiginal-appe}.
By the strict complementarity assumption, $\bar{\nu_i}>0$ for $i\in \mathcal{A}[\bar{y}]$  and $\bar{y}_i>0$ for $i\notin \mathcal{A}[\bar{y}]$. Hence, for $k$ sufficiently large, we have $\mathcal{A}[y^k_+(z^k)]=\mathcal{A}[\bar{y}]$.
Similarly, when $k$ is sufficiently large, we have
$$\mathcal{A}[y(z^k)]=\mathcal{A}[\bar{y}].$$

\end{proof}

We also write down the KKT conditions for $x^*(z)$ for some $z$.
\begin{eqnarray}\label{KKTforproximal}
&&J^TF(x^*(z))y+p(x^*(z)-z)=0,\\
&&\sum_{i=1}^my_i=1,\\
&&y_i\ge 0,\forall i\in [m]\\
&&\mu-\nu_i=f_i(x), \forall i\in[m],\\
&&\nu_i\ge 0, \nu_iy_i=0, \forall i\in[m],
\end{eqnarray}
\begin{lemma}\label{inverse}
If the strict complementarity assumption holds for \eqref{minmax2} , there exists $\delta>0$, such that if
$$\|z\|\le R(x^0, y^0, z^0),$$
and
$$\max\{\|x-x_+(y, z)\|, \|y-y_+(z)\|, \|x_+(y, z)-z\|\}<\delta$$
we have
$$\mathrm{dist}(y_+(z), y(z))< \lambda\|x^*(z)-x(y_+(z), z)\|$$ for some constant $\lambda>0$.
\end{lemma}
\begin{proof}
By Lemma \ref{same active set}
, if the strict complementarity assumption holds for \eqref{minmax2} , there exists $\delta>0$, such that if
$$\|z\|\le R(x^0, y^0, z^0),$$
and
$$\max\{\|x-x_+(y, z)\|, \|y-y_+(z)\|, \|x_+(y, z)-z\|\}<\delta,$$
we have
$$\mathcal{A}[y_+(z)]=\mathcal{A}[y(z)],$$
for some $y(z)\in Y(z)$.
Hence, we have
$$\mathcal{T}(x^*(z))=\mathcal{T}(x(y_+(z), z)).$$
Let $\mathcal{T}=\mathcal{T}(x^*(z))$. Then for $i\notin \mathcal{T}$, $y_i(z)=(y_+(z))_i=0$ and $\|y(z)-y_+(z)\|=\|(y(z))_{\mathcal{T}}-(y_+(z))_{\mathcal{T}}\|$.
Using the optimality conditions for $x(y_+(z), z)$ \eqref{KKT1} and $x^*(z)$ \eqref{KKTforproximal}, we have
\begin{equation}\label{linear-equation}
M^T(x(y_+(z), z))(y_+(z))_{\mathcal{T}}+\begin{Bmatrix}
p(x(y_+(z), z)-z)\\
0
\end{Bmatrix}=(0, 0, \cdots, 0, 1),
\end{equation}
and
\begin{equation}\label{linear-equation2}
M^T(x^*(z))(y(z))_{\mathcal{T}}+\begin{Bmatrix}
p(x^*(z)-z)\\
0
\end{Bmatrix}=(0, 0, \cdots, 0, 1).
\end{equation}
Note that \eqref{linear-equation} can be written as
\begin{equation}\label{linear-equation3}
M^T(x^*(z))(y_+(z))_{\mathcal{T}}=M^T(x^*(z))(y_+(z))_{\mathcal{T}}-M^T(x(y_+(z), z))(y_+(z))_{\mathcal{T}}-\begin{Bmatrix}
p(x(y_+(z), z)-z)\\
0
\end{Bmatrix}.
\end{equation}
By \eqref{linear-equation2} and \eqref{linear-equation3}, we have
$$M^T(x^*(z))((y(z))_{\mathcal{T}}-(y_+(z))_{\mathcal{T}})=(M^T(x(y_+(z), z))-M^T(x^*(z)))(y_+(z))_{\mathcal{T}}-\begin{Bmatrix}
p(x(y_+(z), z)-x^*(z))\\
0
\end{Bmatrix}.$$
Therefore, taking norms to the above and the Lemma \ref{nonsingular}, we have
\begin{eqnarray*}
\gamma\|(y_+(z))_{\mathcal{T}}-(y(z))_{\mathcal{T}}\|&\le&\sqrt{m}L\|x(y_+(z), z)-x^*(z)\|\|(y_+(z))_{\mathcal{T}}\|+p\|x(y_+(z), z)-x^*(z)\|\\
&\le&(\sqrt{m}L+p)\|x^*(z)-x(y_+(z), z)\|,
\end{eqnarray*}
where the first inequality uses the Lipschitz-continuity of $\nabla_xf_i$ and the second is because $\|y_+(z)\|\le 1$.
Hence, we finish the proof with $\lambda=(p+\sqrt{m}L)/\gamma$.
\end{proof}
\begin{proof}[Proof of Lemma \ref{dual error bound-appe}]
By Lemma \ref{weak error bound-appe} and Lemma \ref{inverse}, we have
$$\|x(y_+(z), z)-x^*(z)\|\le \frac{1+\alpha L}{\lambda(p-L)}\|y-y_+(z)\|,$$
which finishes the proof with $\sigma_5=\frac{1+\alpha L}{\lambda(p-L)}$.
\end{proof}
\section{Discussion of the strict complementarity condition}\label{Appendix: sc}
In this section, we discuss some issues about the strict complementarity assumption. First, notice that the min-max problem \eqref{minimax1} and \eqref{minmax2} are both variational inequalities.
As mentioned in the main text of the paper, the strict complementarity assumption is common in the field of variation inequality \cite{harker1990finite, facchinei2007finite}.
 While this assumption is popular, it is still interesting to weaken
 the assumption.
 Inspired by Lemma \ref{uniqueness}, we prove Theorem \ref{discrete} and Lemma \ref{dual error bound} using a weaker regularity assumption rather than the strict complementarity assumption:
\begin{assumption}\label{regularity}
For any $(x^*, y^*)\in W^*$, the matrix $M(x^*)$ is of full column rank.
\end{assumption}
Here recall that
$$M(x^*)=\begin{Bmatrix}
J_{\mathcal{T}(x^*)}&\mathbf{1}
\end{Bmatrix}.$$
We say that Assumption \ref{regularity} is weaker since the strict complementarity assumption (Assumption \ref{sc-appe}) can imply Assumption \ref{regularity} according to Lemma \ref{uniqueness}.
For this assumption, we have the following two claims:
\begin{enumerate}
\item If we replace Assumption \ref{sc-appe} by Assumption \ref{regularity} in Theorem \ref{discrete}, we can attain a same result;
\item In a robust regression problem (will define in \ref{Appendix: rationality of Assumption}), if the data is joint from a continuous distribution, this regularity assumption holds with probability $1$.
\end{enumerate}

\subsection{Replacing {Assumption \ref{sc-appe} by Assumption \ref{regularity} in Theorem \ref{discrete}}}
In this section, we will see that we can prove the dual error bound (Lemma \ref{dual error bound}) using Assumption \ref{regularity} instead of Assumption \ref{sc-appe}.
\begin{lemma}\label{dual2}
Let $$x_+(y, z)=P_X(x-\nabla_xK(x, z; y)).$$
If Assumption \ref{regularity} and the bounded level set assumption hold for \eqref{minmax2} , there exists $\delta>0$, such that if
$$\|z\|\le R(x^0, y^0, z^0),$$
and
$$\max\{\|x-x_+(y, z)\|, \|y-y_+(z)\|, \|x_+(y, z)-z\|\}<\delta$$
we have
$$\|x(y_+(z), z)-x^*(z)\|< \sigma_5\|y-y_+(z)\|$$
for  some constant $\sigma_5>0$.
\end{lemma}
Using this Lemma, we can prove Theorem \ref{discrete} using Assumption \ref{regularity}:
\begin{theorem}\label{discrete2}
Consider solving problem \ref{minmax2} by Algorithm \ref{Alg2} or Algorithm \ref{Alg3}. Suppose that Assumption \ref{regularity} holds and either Assumption \ref{boundedness} holds or assume $\{z^t
\}$ is bounded. Then there exist constants $\beta'$ and $\beta''$ (independent of $\epsilon$ and $T$) such that the following holds:
\begin{enumerate}
    \item (One-block case) If we choose the parameters in Algorithm \ref{Alg2} as in \eqref{eqn: parameters for Alg 2 general case} and further let
$\beta<\beta'$ ,    then
        \begin{enumerate}
        \item Every limit point of $(x^t, y^t)$ is a solution of \eqref{minmax2}.
        \item The iteration complexity of Algorithm   \ref{Alg2} to obtain an $\epsilon$-stationary solution is $\mathcal{O}(1/\epsilon^2)$.
        \end{enumerate}
    \item (Multi-block case) Consider using Algorithm \ref{Alg3} to solve Problem \ref{minmax2}. If we replace the condition for $\alpha$ in \eqref{eqn: parameters for Alg 2 general case} by \eqref{condition}
and require $\beta$ satisfying  $\beta<\epsilon^2$ and $\beta<\beta''$,
    then we have the same results as in the one-block case.
\end{enumerate}
\end{theorem}
\subsection{The rationality of Assumption \ref{regularity}}\label{Appendix: rationality of Assumption}
Intuitively, the assumption \ref{regularity} holds for ``generic problem''.
We rigorously justify this intuition for a simple problem. 
More specifically, we prove that this regularity assumption is generic for a robust regression problem using square loss, i.e., the regularity condition holds with probability $1$ if the outputs of the data points are joint from some continuous distribution.
Consider the following problem:
\begin{equation}\label{regression}
\min_{x\in \mathbb{R}^n}\max_{y\in Y} \frac{1}{2} y_i(\ell_i-\Psi(x, \xi_i))^2,
\end{equation}
where $Y$ is the probability simplex , $\Psi(\cdot)$ is a smooth function used to fit the data (for example the neural network) and $\xi_i, \ell_i$ are the input and the output of the $i$-th data point. We define $\Psi_i(x)=\Psi(x, \xi_i)$ for convenience.
We further make the following mild assumptions:
\begin{assumption}\label{necessary}
$\ell_i$ is joint independently from a continuous distribution over a positive measure set $\mathcal{L}_i\subseteq \mathbb{R}$.
\end{assumption}
Here a continuous distribution over $\mathcal{L}_i$ means that for any zero measure set
$\mathcal{S}\subseteq \mathcal{R}$,
$\mathrm{Pr}(x\in \mathcal{S}\cap \mathcal{L}_i)=0$.
With assumption, for any zero measure set $\mathcal{S}\subseteq \mathbb{R}^m$, $\mathrm{Pr}((\ell_1, \cdots, \ell_m)^T\in \mathcal{S}\cap\prod_{i=1}^m\mathcal{L}_i)=0$.
\begin{assumption}\label{mild}
Let $\Psi(x)=(\Psi_1(x), \cdots, \Psi_m(x))^T$.
Then $\Psi(\mathbb{R}^n)\cap \prod_{i=1}^m\mathcal{L}_i=\Omega$, where $\Omega$ is a zero measure set in $\prod_{i=1}^m\mathcal{L}_i$.
\end{assumption}

This assumption means that $\min_{x}\max_{y\in Y}f_i(x)>0$ with probability $1$. This assumption is reasonable. If there exists an $x^*$ such that $\max_{i}f_i(x^*)=0$, then becaus $f_i(x)\ge 0$, we have $f_i(x^*)=0$ for all $i$.  In this case, we do not need the min-max fomulation! We just need to solve the finite sum problem $\min_{x}\sum_{i=1}^mf_i(x)$. However, in many cases, the uncertainty is large, we do need the robust optimization formulation. So in these cases, Assumption \ref{mild} is reasonable.

Moreover, we have the following lemma:
\begin{lemma}\label{mild1}
Suppose that Assumption \ref{necessary} holds. If $m>n$, Assumption \ref{mild} holds with probability $1$.
\end{lemma}
\begin{proof}
It is direct from the claim that a smooth map $\Psi$ maps a zero measure set into a zero measure set.
Specializing to this lemma, the map $\Psi$ maps $\mathbb{R}^n$ into $\mathbb{R}^m$, hence the image $\Psi(\mathbb{R}^n)$ is of zero measure since $\mathbb{R}^n$ is a zero measure set of $\mathbb{R}^m$. Therefore, $\Psi(\mathbb{R}^n)\cap \prod_{i=1}^m\mathcal{L}_i$ is zero measure in $\mathbb{R}^m$.
\end{proof}
Then we have the following result:
\begin{proposition}\label{generic}
Suppose that Assumption \ref{necessary} and Assumption \ref{mild} hold. Then with probability $1$, every solution of \eqref{regression} satisfies Assumption \ref{regularity}.
\end{proposition}
\subsection{Proof of Lemma \ref{dual2} and Theorem \ref{discrete2}}
For a set $\mathcal{S}\subseteq [m]$, we define
$$M_{\mathcal{S}}(x)=\begin{Bmatrix}
J_{\mathcal{S}F(x; \ell_{\mathcal{S}})} &\mathbf{1}
\end{Bmatrix},$$
where $J_{\mathcal{S}} F(x; \ell_{\mathcal{S}})=((\Psi_i(x)-\ell_i)\nabla_x\Psi_i(x)\mid i\in \mathcal{S})$.

Similar to the proof of Theorem \ref{discrete}, to prove Theorem \ref{discrete2}, we only to prove Lemma \ref{dual2}. Hence, in this section, we only prove Lemma \ref{dual2}.
The proof is similar to the proof of Lemma \ref{dual error bound}. Hence we only give the main steps.
First, similar to Lemma \ref{nonsingular}, we have the following:
\begin{lemma}\label{nonsingular2}
If Assumption \ref{regularity} holds for Problem \eqref{regression} , there exists $\delta>0, \gamma>0$, such that if
$$\|z\|\le R(x^0, y^0, z^0),$$
and
$$\max\{\|x-x_+(y, z)\|, \|y-y_+(z)\|, \|x_+(y, z)-z\|\}<\delta,$$
then
$\gamma(M_{\mathcal{T}(y, z)}(x^*(z)))\ge \gamma$ and $\gamma(M_{\mathcal{T}(y, z)}(x(y_+(z), z)))\ge \gamma$, where
$$\mathcal{T}(y, z)=\mathcal{T}(x^*(z))\cup \mathcal{T}(x(y_+(z), z)).$$
\end{lemma}
\begin{proof}
We prove it by contradiction. Suppose it is not true, there exist $\{x^k\}$, $\{y^k\}\subseteq Y$ and $\{z^k\}\subseteq \mathcal{B}(R(x^0, y^0, z^0))$ such that
$\gamma(M_{\mathcal{T}^k}(x^*(z^k))), \gamma(M_{\mathcal{T}^k}(x(y_+^k(z^k), z^k)))\rightarrow 0$
and
$$\max\{\|x^k-x^k_+(y^k, z^k)\|, \|y^k-y^k_+(z^k)\|, \|x^k_+(y^k, z^k)-z^k\|\}\rightarrow 0,$$
where $\mathcal{T}^k=\mathcal{T}(x^*(z^k))\cup \mathcal{T}(x(y_+^k(z^k), z^k))$.
Since $\mathcal{T}^k$ has only finite choice, without loss of generality, we assume that $\mathcal{T}^k=\mathcal{T}$ for any $k$(passing to a sub-sequence if necessary).
By Lemma \ref{limit}, there exists a $\bar{z}\in X^*$ such that $z^k\rightarrow \bar{z}$.
Hence, by Lemma \ref{Lp_continue} and Lemma \ref{eqvl}, we have
$$x^*(z^k)\rightarrow x^*(\bar{z})=\bar{z}.$$
Therefore by the definition of $\mathcal{T}(x)$, when $k$ is sufficiently large, $\mathcal{T}(x^*(z^k)\subseteq \mathcal{T}(x^*(\bar{z})) =\mathcal{T}(\bar{z}))$.
 Moreover, since $\|y^k-y^k_+(z^k)\|\rightarrow 0$, by  Lemma \ref{weak error bound-appe}, we have
$$\|x(\bar{y}^k_+(z^k), z^k)-x^*(z^k)\|\rightarrow 0.$$
and hence $\mathcal{T}(x(y^k_+(z^k), z^k))\subseteq\mathcal{T}(\bar{z})$.
Then $\mathcal{T}^k\subseteq \mathcal{T}(\bar{z})$ and $\gamma(M)_{\mathcal{T}^k}=0$, which contradicts Assumption \ref{regularity}.
\end{proof}
We then can attain a result similar to Lemma \ref{inverse}.
\begin{lemma}\label{inverse2}
If Assumption \ref{regularity} holds for \eqref{minmax2} , there exists $\delta>0$, such that if
$$\|z\|\le R(x^0, y^0, z^0),$$
and
$$\max\{\|x-x_+(y, z)\|, \|y-y_+(z)\|, \|x_+(y, z)-z\|\}<\delta$$
we have
$$\mathrm{dist}(y_+(z), y(z))< \lambda\|x^*(z)-x(y_+(z), z)\|$$ for some constant $\lambda>0$.
\end{lemma}
\begin{proof}
By Lemma \ref{nonsingular2}, we can find a $\delta>0$ and a $\gamma>0$, such that if
$$\|z\|\le R(x^0, y^0, z^0),$$
and
$$\max\{\|x-x_+(y, z)\|, \|y-y_+(z)\|, \|x_+(y, z)-z\|\}<\delta,$$
then
$\gamma(M_{\mathcal{T}(y, z)}(x^*(z)))\ge \gamma$ and $\gamma(M_{\mathcal{T}(y, z)}(x(y_+(z), z)))\ge \gamma$, where
$$\mathcal{T}(y, z)=\mathcal{T}(x^*(z))\cup \mathcal{T}(x(y_+(z), z)).$$
Let $\mathcal{T}=\mathcal{T}(y, z)$.
Then for $i\notin \mathcal{T}$, $y_i(z)=(y_+(z))_i=0$ and $\|y(z)-y_+(z)\|=\|(y(z))_{\mathcal{T}}-(y_+(z))_{\mathcal{T}}\|$.
Using the optimality conditions for $x(y_+(z), z)$ \eqref{KKT1} and $x^*(z)$ \eqref{KKTforproximal}, we have
\begin{equation}\label{eqn:linear-equation}
M^T_{\mathcal{T}}(x(y_+(z), z))(y_+(z))_{\mathcal{T}}+\begin{Bmatrix}
p(x(y_+(z), z)-z)\\
0
\end{Bmatrix}=(0, 0, \cdots, 0, 1),
\end{equation}
and
\begin{equation}\label{eqn:linear-equation2}
M^T_{\mathcal{T}}(x^*(z))(y(z))_{\mathcal{T}}+\begin{Bmatrix}
p(x^*(z)-z)\\
0
\end{Bmatrix}=(0, 0, \cdots, 0, 1).
\end{equation}
Note that \eqref{eqn:linear-equation} can be written as
\begin{equation}\label{eqn:linear-equation3}
M^T_{\mathcal{T}}(x^*(z))(y_+(z))_{\mathcal{T}}=M^T_{\mathcal{T}}(x^*(z))(y_+(z))_{\mathcal{T}}-M^T_{\mathcal{T}}(x(y_+(z), z))(y_+(z))_{\mathcal{T}}-\begin{Bmatrix}
p(x(y_+(z), z)-z)\\
0
\end{Bmatrix}.
\end{equation}
By \eqref{eqn:linear-equation2} and \eqref{eqn:linear-equation3}, we have
$$M^T_{\mathcal{T}}(x^*(z))(y(z)-y_+(z))=(M^T_{\mathcal{T}}(x(y_+(z), z))-M^T_{\mathcal{T}}(x^*(z)))(y_+(z))_{\mathcal{T}}-p(x(y_+(z), z)-x^*(z)).$$
Therefore, taking norms to the above and the Lemma \ref{nonsingular}, we have
\begin{eqnarray*}
\gamma\|(y_+(z))_{\mathcal{T}}-(y(z))_{\mathcal{T}}\|&\le&\sqrt{m}L\|x(y_+(z), z)-x^*(z)\|\|(y_+(z))_{\mathcal{T}}\|+p\|x(y_+(z), z)-x^*(z)\|\\
&\le&(\sqrt{m}L+p)\|x^*(z)-x(y_+(z), z)\|,
\end{eqnarray*}
where the first inequality uses the Lipschitz-continuity of $\nabla_xf_i$ and the second is because $\|y_+(z)\|\le 1$.
Hence, we finish the proof with $\lambda=(p+\sqrt{m}L)/\gamma$.
\end{proof}
Then Lemma \ref{inverse2} and Lemma \ref{weak error bound} yield Theorem \ref{discrete2}.
\subsection{Proof of Proposition \ref{generic}}
For a set $\mathcal{S}\subseteq [m]$, we define
$$M_{\mathcal{S}}(x; \ell_{\mathcal{S}})=\begin{Bmatrix}
J_{\mathcal{S}F(x; \ell_{\mathcal{S}})} &\mathbf{1}
\end{Bmatrix},$$
where $J_{\mathcal{S}} F(x; \ell_{\mathcal{S}})=((\Psi_i(x)-\ell_i)\nabla_x\Psi_i(x)\mid i\in \mathcal{S})$.

\begin{proof}
Define the event $\mathcal{E}_{\mathcal{T}, \mathcal{P}}$ to be:
there exists a solution $(x^*, y^*)\in W^*$ , such that $M(x^*)$  is not of full row rank, $\mathcal{T}(x^*)=\mathcal{T}$  and $\Psi_i(x^*)-\ell_i\ge0$ for $i\in \mathcal{P}$ and $\Psi_i(x^*)-\ell_i$ for $i\notin \mathcal{P}$.
Then Proposition \ref{generic} is equivalent to the claim:
$$\mathrm{Pr}(\cup_{\mathcal{T}\subseteq [m], \mathcal{P}\subseteq \mathcal{T}}\mathcal{E}_{\mathcal{T}, \mathcal{P}})=0.$$
Since there are only finite choice of the sets $\mathcal{T}$ and $\mathcal{P}$, we only need to prove that for any $\mathcal{T}\subseteq [m]$ and $\mathcal{P}\subseteq \mathcal{T}$, $\mathcal{E}_{\mathcal{T}, \mathcal{P}} $ holds with probability $0$,
Without loss of generality, we let $\mathcal{T}=\{1, 2, \cdots, k\}$ and $\mathcal{P}=\{1, 2,\cdots, p\}$ with $p\le k$.
We define $\delta_i$ for $i\in [k]$  as $\delta_i=1$ for $i\in \mathcal{P}$ and $\delta_i=-1$ otherwise.
Then if $\mathcal{E}_{\mathcal{T}, \mathcal{P}}$ holds, there exists an $x^*=(x_1^*, \cdots, x_{n}^*)^T\in X^*$ and $x_{n+1}\in \mathbb{R}$, such that

\begin{enumerate}
\item $(x_1^*, \cdots, x_n^*)^T\in X^*$;
\item $x_{n+1}^*\ge 0$;
\item $\mathcal{T}(x^*)=\mathcal{T}$;
\item $\Psi_i(x^*)-\ell_i=x_{n+1}^*\ge 0$ for $i\in \mathcal{P}$ and $\Psi_i(x^*)-\ell_i=-x^*_{n+1}\le 0$ for $i\notin \mathcal{P}$.
\item $M_{\mathcal{T}}(x_1^*, \cdots, x_n^*; \ell_1, \cdots, \ell_k)$ is row rank deficient.
\end{enumerate}
Define $\bar{X}^*_{\mathcal{T}, \mathcal{P}}(\ell_1, \cdots, \ell_k)$ to be the set of all $x^*\in X^*$ satisfying the  above conditions.
Consider the map $G: \mathbb{R}^{n+1}\rightarrow \mathbb{R}^k$ defined as
$$G(x_1, \cdots, x_{n+1})=(\Psi_1(x_1, \cdots, x_n)-\delta_1x_{n+1}, \cdots, \Psi_k(x_1, \cdots, x_n)-\delta_kx_{n+1})^T.$$
Then $G(x_1^*, \cdots, x_{n+1}^*)=(\ell_1, \cdots, \ell_k)$ for any $(x_1^*, \cdots, x_{n+1}^*)^T\in \bar{X}^*_{\mathcal{T}, \mathcal{P}}(\ell_1, \cdots, \ell_k)$.
Define the set $\bar{X}_{\mathcal{T}, \mathcal{P}}\subseteq \mathbb{R}^{n+1}$ be the collection of all $(x_1, \cdots, x_{n+1})$ satisfying:
\begin{enumerate}
\item $x_{n+1}> 0$.
\item there exist $\bar{\ell}_1, \cdots, \bar{\ell}_k$ with $\Psi_i(x_1, \cdots, x_n)-\bar{\ell}_i=x_{n+1}$ for $i\in
\mathcal{P}$ and $\Psi_i(x_1, \cdots, x_n)-\ell_i=-x_{n+1}$ for $i\notin \mathcal{P}$.
\item $M_{\mathcal{T}}(x_1, \cdots, x_n; \bar{\ell}_1, \cdots, \bar{\ell}_k)$ is rank deficient.
\item $(\bar{\ell_1}, \cdots, \bar{\ell}_k)^T\in \prod_{i=1}^k\mathcal{L}_i$.
\end{enumerate}
Therefore, if $\mathcal{E}_{\mathcal{T}, \mathcal{P}}$ holds, we have
$$(\ell_1, \cdots, \ell_m)^T\in (G(\bar{X}_{\mathcal{T}, \mathcal{P}})\cap \prod_{i=1}^k\mathcal{L}_i)\times \prod_{i=k+1}^m\mathcal{L}_i\cup \Omega.$$
For $(x_1, \cdots, x_{n+1})^T\in \bar{X}_{\mathcal{T}, \mathcal{P}}$,
notice that $J G(x_1, \cdots, x_{n+1})$ is attained by doing elementary matrix transformation to the matrix $M_{\mathcal{T}}(x_1, \cdots, x_n; \bar{\ell}_1, \cdots, \bar{\ell}_k)$, i.e., multiplying the first $k$ columns of $M_{\mathcal{T}}(x_1, \cdots, x_n; \bar{\ell}_1, \cdots, \bar{\ell}_k)$ by $1/x_{n+1}$ and multiplying the $k+1$-th column of $M_{\mathcal{T}}(x_1, \cdots, x_n; \bar{\ell}_1, \cdots, \bar{\ell}_k)$ by $-1$ and then multiplying the $i$-th row by $\delta_i$ for $i\in [n]$. Therefore, $M_{\mathcal{T}}(x_1, \cdots, x_n; \bar{\ell}_1, \cdots, \bar{\ell}_k)$ is also rank deficient.

Consequently, $G(x_1, \cdots, x_{n+1})$ with $(x_1, \cdots, x_{n+1})^T\in \bar{X}_{\mathcal{T}, \mathcal{P}}$ is a critic value of $G$ (see \cite{berger2012differential}). Then by Sard's Theorem \cite{berger2012differential}, $G(\bar{X}_{\mathcal{T}, \mathcal{P}})$ is a zero measure set in $\mathbb{R}^k$.
Hence, $G(\bar{X}_{\mathcal{T}, \mathcal{P}})\cap \prod_{i=1}^k\mathcal{L}_i$ is a zero measure set in $\prod_{i=1}^k\mathcal{L}_i$.
Recall that if $\mathcal{E}_{\mathcal{T},\mathcal{P}}$ holds, we have
$$(\ell_1, \cdots, \ell_m)^T\in \mathcal{Z}= (G(\bar{X}_{\mathcal{T}, \mathcal{P}})\cap \prod_{i=1}^k\mathcal{L}_i)\times \prod_{i=k+1}^m\mathcal{L}_i\cup \Omega.$$
By the above analysis, $G(\bar{X}_{\mathcal{T}, \mathcal{P}})\cap \prod_{i=1}^k \mathcal{L}_i$ is a zero measure set in $\prod_{i=1}^k\mathcal{L}_i$. Hence, $(G(\bar{X}_{\mathcal{T}, \mathcal{P}})\cap \prod_{i=1}^k\mathcal{L}_i)\times \prod_{i=k+1}^m\mathcal{L}_i$ is a zero measure set in $\prod_{i=1}^m\mathcal{L}_i$. Also by Assumption \ref{mild}, $\Omega$ is a zero measure set in $\prod_{i=1}^m\mathcal{L}_i$. Consequently, $\mathcal{Z}$ is a zero measure set in $\prod_{i=1}^m\mathcal{L}_i$.
 Then by the continuity of the distribution of $\ell$, we finish the proof.
\end{proof}

\section{Proof of Theorem \ref{general-for-optimization-stationary-point}}\label{Appendix: new theorem}

\begin{proof}
    According to the two inequalities \ref{key1} and \ref{beta-bound},
the weaker bound Lemma \ref{weak error bound} and Lemma \ref{Lp_continue},
taken $\beta = \epsilon^2$, we have
\begin{equation}\label{v2-key4-appe}
\max\{\|x^*(z^t)-x(y^t_+(z^t), z^t)\|, \|x^{t+1}-x(y^t_+(z^t), z^t)\|, \|z^t-x^{t+1}\|\}\le \mathcal{O}(\epsilon),
\end{equation}
then by triangle inequality,
\begin{equation}\label{final-equa}
\|z^t - x^*(z^t)\| \le \mathcal{O}(\epsilon).
\end{equation}
\end{proof}

\section{Details in Experiments}\label{Appendix: experiment}
Recall the procedure of training a robust neural network against adversarial attacks can be formulated as a min-max problem:
\begin{align}
\label{eq: Madry2 appendix}
\min_{\mathbf{w}}\, \; \sum_{i=1}^{N}\;\max_{\delta_i,\;\text{s.t.}\;|\delta_i|_{\infty}\leq \varepsilon}   {\ell}(f(x_i+\delta_i;\mathbf{w}), y_i),
\end{align}
where $\mathbf{w}$ is the parameter of the neural network, the pair $(x_i,y_i)$ denotes the $i$-th data point,
and $\delta_i$ is the perturbation added to data point~$i$.

As \eqref{eq: Madry2 appendix} is nonconvex-nonconcave and thus difficult to solve directly,
researchers introduce an approximation of \eqref{eq: Madry2 appendix} \cite{nouiehed2019solving}
where the approximated problem has a concave inner problem.
The approximation is first replacing the inner maximization problem in \ref{eq: Madry2 appendix} with a finite max problem:
\begin{equation} \label{eq:adversaryxhatformulation appendix}
\min_{\mathbf{w}}\, \; \sum_{i=1}^{N}\;\max \left\{   {\ell}(f(\hat{x}_{i0}(\mathbf{w});\mathbf{w}), y_i), \ldots,{\ell}(f(\hat{x}_{i9}(\mathbf{w});\mathbf{w}), y_i) \right\},
\end{equation}
where each $\hat{x}_{ij}(\mathbf{w})$ is the result of a targeted attack on sample $x_i$ by changing the output of the network to label~$j$.

To obtain the targeted attack $\hat{x}_{ij}(\mathbf{w})$, we need to introduce an additional procedure.
Recall the images in MNIST have 10 classifications,
thus the last layer of the neural network architecture for learning classification have 10 different neurons.
To obain any targeted attack $\hat{x}_{ij}(\mathbf{w})$, we perform gradient ascent for $K$ times:
$$
    x_{ij}^{k+1} = \text{Proj}_{B(x, \varepsilon)}\Big[x_{ij}^k + \alpha \nabla_x(Z_{j}(x_{ij}^k, \mathbf{w})-Z_{y_i}(x_{ij}^k, \mathbf{w}))\Big],\; k=0, \cdots, K-1,
$$

and let $ \hat{x}_{ij}(\mathbf{w})= x_{ij}^K$. Here, $Z_j$ is the network logit before softmax corresponding to label $j$; $\alpha>0$ is the step-size; and $\text{Proj}_{B(x, \varepsilon)}[\cdot]$ is the projection to the infinity ball with radius $\varepsilon$ centered at $x$.  Using the same setting in \cite{nouiehed2019solving}, we set the iteration number as $K = 40$, the stepsize as $\alpha = 0.01$, and the perturbation level $\epsilon$ chosen from $\{0.0, 0.1, 0.2, 0.3, 0.4\}$.

Now we can replace the finite max problem~\eqref{eq:adversaryxhatformulation appendix} with a concave problem over a probabilistic simplex, where the entire problem is non-convex in~$w$, but concave in~$\mathbf{t}$:
\begin{align}
\label{eq:finite-max2-appe}
\min_{\mathbf{w}}\;\sum_{i=1}^{N}\;\max_{\mathbf{t} \in \mathcal{T}}\sum_{j=0}^{9}\;{t_j}\ell\left(f\left(x_{ij}^K;\mathbf{w}\right),y_i\right),\; \mathcal{T} =\{(t_1, \cdots, t_m)\mid \sum_{i=1}^mt_i=1, t_i\ge 0\}.
\end{align}

We use Convolutional Neural Network(CNN) with the architecture detailed in Table~\ref{tab:net-arch-robust} in the experiments. This setting is the same as in \cite{nouiehed2019solving}.

The results are listed in Table \ref{tab:robust_nn_results}. The first three lines are the results obtained from \cite{nouiehed2019solving} and the fourth line is obtained by using the code provided in  \cite{nouiehed2019solving} to train their algorithm. As for comparison, we run our algorithm \ref{Alg2} for the same number of iterations (100 iterations) with parameter $p= 0.2, \beta = 0.8$ and $\alpha = 0.5$. 
In the experiment, to compute the projection of a vector of dimension $d$ over the probability simplex, we use the algorithm from \cite{wang2013projection} which has a complexity $\mathcal{O}(d \log d)$.

We also perform robust training on CIFAR10 \cite{krizhevsky2009learning} and comparing with the algorithm in \cite{nouiehed2019solving} after 30 epochs. As shown in Figure \ref{fig: cifar10}, our algorithm still outperforms \cite{nouiehed2019solving} in convergence speed. To obtain the targeted attack at each epoch, we set the iteration number $K$ as 10, the stepsize as $0.007$, and the perturbation level as 0.031 which are the same settings appear in \cite{zhang2019theoretically}. We achieve robust accuracy $38.5\%$ and testing accuracy $82.6\%$ which are comparable to the results from the literature \cite{wong2018scaling} in robust training.

\newpage

\begin{table}[!ht]
\centering

\begin{tabular}{@{}ll@{}}
\toprule
Layer Type           & Shape               \\ \midrule
Convolution $+$ ReLU & $5 \times 5 \times 20$     \\
Max Pooling          & $2 \times 2$                \\
Convolution $+$ ReLU & $5 \times 5 \times 50$ \\
Max Pooling          & $2 \times 2$             \\
Fully Connected $+$ ReLU & $800$ \\
Fully Connected $+$ ReLU & $500$ \\
Softmax & $10$ \\ \bottomrule
\end{tabular}
\vskip 0.1in
\caption{Model Architecture for the MNIST dataset.}
\label{tab:net-arch-robust}
\end{table}

\begin{figure}[!h]
    \centering
    \includegraphics[width=0.4\textwidth]{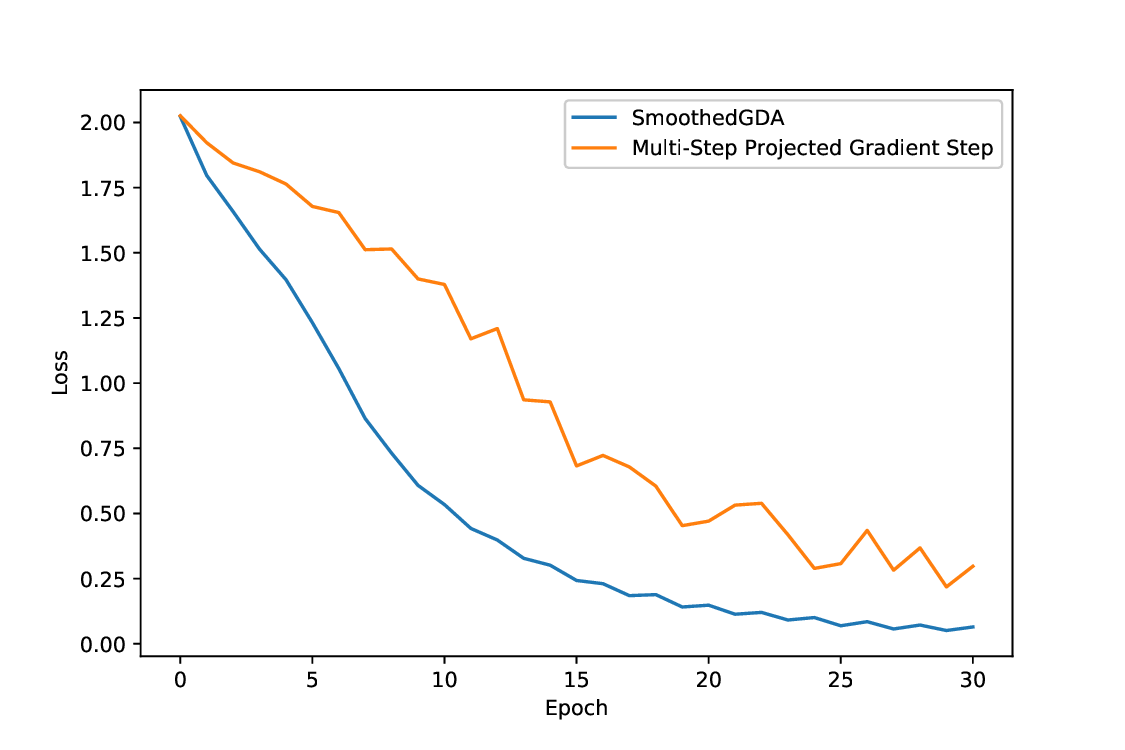}
    \caption{Convergence speed of Smoothed-GDA and the algorithm in \cite{nouiehed2019solving} on CIFAR10.
    }
    \label{fig: cifar10}
\end{figure}



\end{document}